\documentclass[12pt,centertags]{amsart}
\usepackage{amssymb}
\usepackage{amsmath,amstext,amsthm,a4,amssymb,amscd}
\usepackage{charter}
\usepackage{typearea}
\usepackage[mathscr]{eucal}
\usepackage{mathrsfs}
\usepackage{epsf}
\usepackage{extarrows}
\usepackage{hyperref}
\usepackage[a4paper,top=3cm,bottom=3cm,left=2cm,right=2cm]{geometry}
\numberwithin{equation}{section}
\allowdisplaybreaks[3]

\newcommand{\field}[1]{\mathbb{#1}}

\newcommand{\R}{\field{R}}
\newcommand{\C}{\field{C}}
\newcommand{\N}{\field{N}}

 \def\cC{\mathscr{C}}
 \def\cF{\mathscr{F}}
\def\cL{\mathscr{L}}
\def\cJ{\mathscr{J}}

\def\cO{\mathscr{O}}
\def\cP{\mathscr{P}}

\def\mH{\mathcal{H}}

\def\mJ{\mathcal{J}}

\def\mO{\mathcal{O}}

\def\mR{\mathcal{R}}

\newcommand{\bb}{\boldsymbol{b}}

\DeclareMathOperator{\End}{End}

\DeclareMathOperator{\Ker}{Ker}

\DeclareMathOperator{\td}{Td}

\newcommand{\om}{\omega}

\newtheorem{thm}{Theorem}[section]
\newtheorem{lemma}[thm]{Lemma}

\theoremstyle{definition}

\theoremstyle{definition}

\newcommand{\be}{\begin{eqnarray}}
\newcommand{\ee}{\end{eqnarray}}
\newcommand{\ov}{\overline}

\newcommand{\p}{\partial}
\newcommand{\comment}[1]{}

\begin{document}

\title[Optimal convergence speed of Bergman metrics]
{Optimal convergence speed of Bergman metrics on
symplectic manifolds}

\author{Wen Lu}

\address{School of Mathematics and Statistics,
Huazhong University of Science and Technology,
Wuhan 430074, China}
\email{wlu@hust.edu.cn}
\thanks{W.\ L.\ supported by National Natural Science Foundation
of China (Grant Nos. 11401232, 11871233)}

\author{Xiaonan Ma}

\address{
Universit{\'e} de Paris \&Universit{\'e} Paris Diderot - Paris 7,
UFR de Math{\'e}matiques, Case 7012,
75205 Paris Cedex 13, France}
\email{xiaonan.ma@imj-prg.fr}
\thanks{X.\ M.\ partially supported by
NNSFC No. 11829102 and
funded through the Institutional Strategy of
the University of Cologne within the German Excellence Initiative}
\author{George Marinescu}

\address{Univerisit\"at zu K\"oln, Mathematisches institut,
Weyertal 86-90, 50931 K\"oln, Germany
\&Institute of Mathematics 'Simion Stoilow',
Romanian Academy, Bucharest, Romania}
\email{gmarines@math.uni-koeln.de}
\thanks{G.\ M.\ partially supported by DFG funded
project SFB TRR 191}

\keywords{Symplectic Kodaira embedding,
Bergman metric, Bergman kernel asymptotics,
renormalized Bochner Laplacian}
\subjclass[2010]{58J60,  53D50}

\date{\today}

\begin{abstract}
It is known that a compact symplectic manifold endowed with
a prequantum line bundle can be embedded in the projective
space generated by
the eigensections of low energy of the Bochner Laplacian
acting on high $p$-tensor powers of the prequantum
line bundle. We show that the Fubini-Study forms induced by
these embeddings converge
at speed rate $1/p^{2}$ to the symplectic form.
This result implies the generalization to the almost-K\"ahler case of
the lower bounds on the Calabi functional given by Donaldson for
K\"ahler manifolds,
as shown by Lejmi and Keller.
\end{abstract}

\maketitle
\setcounter{section}{-1}

\section{Introduction} \label{s0}
A very useful tool in the study of canonical K\"ahler metrics
is the use of Bergman
metrics to approximate arbitrary K\"ahler metrics in a given
integral cohomology class, see e.\,g., \cite{Do01,Ma10,Tian}.

Let $(X,\om)$ be a compact K\"ahler manifold endowed with
a Hermitian holomorphic line bundle $(L,h^L)$
such that $\frac{\sqrt{-1}}{2\pi}R^{L}=\omega$.
Since the bundle $L$ is positive, Kodaira's theorem shows
that high powers $L^p$ give rise to
holomorphic embeddings $\Phi_p:X\to\mathbb{P}(H^0(X,L^p)^*)$.
The \emph{Bergman form} $\omega_p$ at level $p$
is defined as the rescaled induced Fubini-Study form
$\frac 1p \Phi_p^*\omega_{_{\rm FS}}$,
where $\omega_{_{\rm FS}}$
is the natural Fubini-Study form on $\mathbb{P}(H^0(X,L^p)^*)$.
Tian \cite{Tian} showed that $\omega_p$
converges to $\om$ in the $\cC^2$ topology with speed rate
$p^{-1/2}$, as $p\to\infty$, that is, there exists $C>0$ such that
for any $p\in\N^*$ we have
\begin{equation}\label{0.6a}
\Big|\frac{1}{p}\Phi_{p}^{\ast}(\omega_{_{\rm FS}})
-\omega\Big|_{\cC^2}\leqslant  \frac{C}{p^{1/2}}\,\cdot
\end{equation}
 This was improved by Ruan \cite{Ruan98} to convergence in
 $\cC^\infty$ with speed rate $p^{-1}$
(see also \cite[Theorem 5.1.4]{MM07}).
Tian's result was motivated by a problem of Yau \cite{Yau87}.

The process described above can be seen in the general framework
of quantization. The Bergman forms $\omega_p$ can be thought as
quantization at level $p$ of the original K\"ahler form $\omega$.
The number $1/p$ is to be thought of as analogous to
Planck's constant and in the
semiclassical limit $p\to\infty$ the quantized objects $\omega_p$
converge to the original K\"ahler one.

The proof of the convergence in \cite{Ruan98,Tian} is based
on the diagonal expansion of the
Bergman kernel up to second order.
A full diagonal asymptotic expansion of the Bergman kernel
in powers of $p$ in the $\cC^\infty$
topology was obtained by Catlin \cite{Catlin} and
Zelditch \cite{Zelditch98} as an application of
Boutet de Monvel and Sj\"{o}strand's work \cite{BouSt76},
see also \cite{DLM06,MM08a} for different approaches and
generalizations.
We refer to \cite{MM07} for a comprehensive study of several
analytic and geometric aspects of Bergman kernel.
One advantage of the expansion in the
$\cC^\infty$ topology is that it easily implies the convergence
of the Bergman forms $\om_p$
to $\om$ with speed rate $p^{-2}$, see \cite[(5.1.23)]{MM07}.
This convergence speed is optimal.
Note that the scalar curvature is
up to a multiplicative constant the
coefficient of the second term of the Bergman kernel expansion.
The purpose of this paper is to extend this optimal
result to the case of symplectic manifolds.

The Bergman kernel of a holomorphic line bundle $L$ on
a complex manifold is the smooth kernel of the orthogonal
projection from the space of square integrable sections
on the space of holomorphic sections,
or, equivalently, on the kernel of the Kodaira Laplacian
$\Box^L= \ov{\p}^{L}\,\ov{\p}^{L*}
+ \ov{\p}^{L*}\,\ov{\p}^{L}$ on $L$.
In order to find a suitable notion
of ``holomorphic section'' of a prequantum line bundle on
a compact \emph{symplectic manifold},
Guillemin and Uribe \cite{GU} introduced a renormalized
Bochner Laplacian $\Delta_{p, 0}$ (cf.\,\eqref{0.4a})
which reduces to $2\Box ^L$ in the K\"ahler case.

We describe this construction in detail. Let $(X, \omega)$ be
a compact symplectic
manifold of real dimension $2n$. Let $(L, h^{L})$ be a
Hermitian line bundle on $X$, and let $\nabla^{L}$ be
a Hermitian connection on $(L, h^{L})$ with
the curvature $R^{L}=(\nabla^{L})^{2}$. We will assume throughout
the paper that $(L, h^{L},\nabla^{L})$ is a prequantum line bundle
 of $(X, \omega)$, i.e.,   
\begin{equation}\label{0.1}
\frac{\sqrt{-1}}{2\pi}R^{L}=\omega.
\end{equation}
We choose an almost complex structure $J$ such that
$\omega$ is $J$-invariant and $\omega(\cdot, J\cdot)>0$.
The almost complex structure $J$
induces a splitting
$TX\otimes_{\R}\C=T^{(1, 0)}X\oplus T^{(0, 1)}X$,
where $T^{(1, 0)}X$ and $T^{(0, 1)}X$
are the eigenbundles of $J$ corresponding to the eigenvalues
$\sqrt{-1}$ and $-\sqrt{-1}$, respectively.

Let $g^{TX}(\cdot, \cdot):=w(\cdot, J\cdot)$ be the
Riemannian metric on $TX$ induced by $\omega$ and $J$.
The Riemannian volume form $dv_{X}$ of $(X, g^{TX})$ has
the form $dv_{X}=\omega^{n}/n!$. The
$L^{2}$-Hermitian product on the space $\cC^{\infty}(X, L^{p})$
of smooth sections
of $L^{p}$ on $X$, with $L^{p}:=L^{\otimes p}$, is given by
\begin{equation}\label{lus0.2}
\big\langle s_{1}, s_{2}\big\rangle
=\int_{X}\big\langle s_{1}, s_{2}\big\rangle(x)dv_{X}(x).
\end{equation}
Let $\nabla^{TX}$ be the Levi-Civita connection on $(X, g^{TX})$
with curvature $R^{TX}$, and let $\nabla^{L^{p}}$
be the connection on $L^{p}$ induced by $\nabla^{L}$.
Let $\{e_{k}\}$ be a local orthonormal frame of $(TX, g^{TX})$.
The Bochner Laplacian acting on $\cC^{\infty}(X, L^{p})$ is given by
\begin{equation}
\Delta^{L^{p}}=-\sum_{k}\Big[\big(\nabla_{e_{k}}^{L^{p}}\big)^{2}
-\nabla_{\nabla_{e_{k}}^{TX}e_{k}}^{L^{p}}\Big].
\end{equation}
Given $\Phi\in \cC^{\infty}(X, \R)$, the renormalized
Bochner Laplacian is defined by
\begin{align}\label{0.4a}
\Delta_{p, \Phi}=\Delta^{L^{p}}-2\pi np+\Phi.
\end{align}
By \cite{GU}, \cite[Corollary\,1.2]{MM02}, there exists
$C_{L}>0$ independent of $p$ such that
\begin{align}\label{0.4}
\textup{Spec}(\Delta_{p, \Phi})\subset [-C_{L}, C_{L}]
\cup [4\pi p-C_{L}, +\infty),
\end{align}
where $\textup{Spec}(A)$ denotes the spectrum of the operator
$A$. Since $\Delta_{p, \Phi}$ is an elliptic operator on
a compact manifold, it has discrete spectrum and its eigensections
are smooth. Let $\mH_{p}$
be the direct sum of eigenspaces of $\Delta_{p, \Phi}$
corresponding to the
eigenvalues lying in $[-C_{L}, C_{L}]$.
In mathematical physics terms, the operator $\Delta_{p, \Phi}$
is a semiclassical Schr\"odinger
operator and the space $\mH_p$ is the space of its bound states
as $p\to\infty$. The space $\mH_p$ proves to be
an appropriate replacement for the space of holomorphic sections
$H^0(X,L^p)$ from the K\"ahler case. In particular, we have
for $p$ large enough (cf.\ \cite[(8.3.3)]{MM07}),
\begin{align}\label{0.4b}
\dim \mH_{p}=\int_{X}\td(T^{(1, 0)}X)e^{p\omega},
\end{align}
where
$\td(T^{(1, 0)}X)$ is the Todd class of $T^{(1, 0)}X$,
which corresponds to the Riemann-Roch-Hirzebruch
 formula from complex geometry.

Let $\mathbb{P}(\mH^{\ast}_{p})$ be the projective space
associated to the dual space
of $\mH_{p}$; we identify $\mathbb{P}(\mH^{\ast}_{p})$
with the Grassmannian of hyperplanes in $\mH_{p}$.
The base locus of $\mH_{p}$ is the set
$\textup{Bl}(\mH_{p})=\big\{x\in X: s(x)=0$
for all $s\in \mH_{p}\big\}$. We define the Kodaira map
\begin{equation}\label{0.5}
\Phi_{p}: X  \backslash  \textup{Bl}(\mH_{p})\rightarrow
\mathbb{P}(\mH^{\ast}_{p}),
\quad \Phi_{p}(x) =\big\{s\in \mH_{p}: s(x)=0\big\}\,,
\end{equation}
which sends $x\in X\backslash \textup{Bl}(\mH)$ to the hyperplane
of sections vanishing at $x$.
Note that $\mH_{p}$ is endowed with the induced
$L^{2}$ Hermitian product (\ref{lus0.2}) so there is a well-defined
Fubini-Study metric $g_{_{\rm FS}}$ on
$\mathbb{P}(\mH^{\ast}_{p})$ with the associated form
$\omega_{_{\rm FS}}$.

The symplectic Kodaira embedding theorem
\cite[Theorem\,3.6]{MM08a}, \cite[Theorem\,8.3.12]{MM07},
states that for large $p$ the Kodaira maps
$\Phi_{p}: X\rightarrow \mathbb{P}(\mH^{\ast}_{p})$ are
embeddings and the Bergman forms converge to
the symplectic form with speed rate $p^{-1}$.
We note that in this case the
near-diagonal expansion of the Bergman kernel is essential for
the proof, in contrast to the the K\"ahler case,
where the diagonal expansion already implies the result.
Let us also observe that
\cite[Theorem\,3.6]{MM08a} and \cite[Theorem\,8.3.12]{MM07}
are valid in a more general context,
namely when $g^{TX}$ is an arbitrary $J$-invariant Riemannian metric.

There exists in the literature another replacement of the notion
of holomorphic section, see e.\,g., \cite{BorU00,ShZe02}. It is
based on a construction by
Boutet de Monvel and Guillemin \cite{BouGou81}
of a first-order pseudodifferential
operator $D_b$ 
on the circle bundle of  $L^*$.
 The associated Szeg\H{o} kernels are well defined
modulo smooth operators on the associated circle bundle,
even though $D_b$ is neither canonically defined nor unique.
Indeed, Boutet de Monvel--Guillemin define the Szeg\H{o} kernels first,
and construct the operator $D_b$ from the Szeg\H{o} kernels.
For these spaces the Bergman forms converge to
the symplectic form with speed rate $p^{-1}$, too.

The main result of this paper is as follows.
\begin{thm}\label{t0.1}
Let $(X,\om)$ be a compact symplectic manifold and $(L,h^L)$
be a Hermitian line bundle endowed with
a Hermitian connection $\nabla^{L}$ such that
$\frac{\sqrt{-1}}{2\pi}R^{L}=\omega$ holds.
Let $J$ be an almost complex structure on $TX$ such that
$g^{TX}(\cdot, \cdot):=\om(\cdot, J\cdot)$ is a  $J$-invariant
Riemannian metric on $TX$.
Then for any $\ell\in\N$, there exists $C_{\ell}>0$ such that
\begin{equation}\label{0.6}
\Big|\frac{1}{p}\Phi_{p}^{\ast}(\omega_{_{\rm FS}})
-\omega\Big|_{\cC^{\ell}}\leqslant  \frac{C_{\ell}}{p^{2}}\,,
\end{equation}
where $\Phi_p$ is the Kodaira map \eqref{0.5}
defined by the space $\mH_p$ of bound states
of the renormalized Bochner Laplacian
$\Delta_{p, \Phi}$ 
associated with $g^{TX}, \nabla^L,\Phi$ in (\ref{0.4a}).
\end{thm}
The proof is based on the near diagonal expansion of the
Bergman kernel of $\mH_p$ from \cite{MM07,MM08a}.
The sharp bound of $\cO(p^{-2})$ is due to some remarkable
cancellations of the coefficients in this expansions,
reminiscent of the local properties of the curvature of
K\"ahler metrics.

The main motivation for approximating K\"ahler metrics
by Fubini-Study metrics arises from questions about the
existence and uniqueness of K\"ahler metrics of
constant scalar curvature, or more generally,
K\"ahler-Einstein metrics, see \cite{Do01,Do05,Tian,Yau87}.
It is natural to study such questions also in the symplectic framework,
for example, it is interesting to generalize to the almost-K\"ahler case
the lower bounds on the
Calabi functional given by Donaldson \cite{Do05}.
This is done by Lejmi and Keller \cite{LK}.
Theorem \ref{t0.1} plays a crucial role in their proof in the symplectic case.

The organization of this paper is as follows. In Section \ref{s1},
we recall the formal calculus
on $\C^{n}$ for the model operator $\cL$ (cf. (\ref{lus1.1a})),
which is the main ingredient of our approach.
In Section \ref{s2}, we review the asymptotic expansion
of the generalized Bergman
kernel. In Section \ref{s3}, we reduce the proof of
Theorem \ref{t0.1} to Theorem \ref{t3.1}.
In Section \ref{s4}, we prove Theorem \ref{t3.1}
and thus finish the proof of Theorem \ref{t0.1}.

We shall use the following notations.
For $\alpha=(\alpha_{1}, \ldots, \alpha_{n})\in\N^{n}$,
$z\in \C^{n}$, we set $|\alpha|=\sum^{n}_{j=1}\alpha_{j}$,
$\alpha!=\prod_{j}(\alpha_{j}!)$
and $z^{\alpha}:=z_{1}^{\alpha_{1}}\cdots z^{\alpha_{n}}_{n}$.
Moreover, when an index variable appears twice in a single term,
it means that we are summing over all
its possible values.

\medskip

\noindent{\bf Acknowledgments}. We would like to thank
Mehdi Lejmi, Julien Keller and  G\'abor Sz\'ekelyhidi for motivating
and helpful discussions which led to the writing of this paper.

\section{Kernel calculus on $\C^{n}$}\label{s1}

In this section, we recall the formal calculus on $\C^{n}$ for
the model operator $\cL$ introduced
in \cite[\S\,1.4]{MM08a}, \cite[\S\,4.1.6]{MM07} (with $a_{j}=2\pi$
therein).
This calculus is the main ingredient of our approach.

Let us consider the canonical coordinates $(Z_{1}, \ldots, Z_{2n})$
on the real vector space $\R^{2n}$. On the
complex vector space $\C^{n}$ we consider the complex
coordinates $(z_{1}, \ldots, z_{n})$.
The two sets of coordinates are linked by the relation
$z_{j}=Z_{2j-1}+\sqrt{-1}Z_{2j}$, $j=1, \ldots, n$.

We consider the $L^{2}$-norm
\begin{align}\label{luw1.1}
\|\boldsymbol\cdot\|_{L^{2}}
=\Big(\int_{\R^{2n}}|\,\boldsymbol\cdot\,|^{2}dZ\Big)^{1/2}\
\textup{on}\ \R^{2n},
\end{align} where
$dZ=dZ_{1}\ldots dZ_{2n}$ is the Lebesgue measure.
We define the differential operators:
\begin{align}\label{lus1.1a}
b_{j}=-2\frac{\p}{\p z_{j}}+\pi\ov{z}_{j},\ \
b^{+}_{j}=2\frac{\p}{\p \ov{z}_{j}}+\pi z_{j},
\ \ b=(b_{1}, \ldots, b_{n}),\ \ \cL=\sum^{n}_{j=1}b_{j}b^{+}_{j},
\end{align}
which extend to closed densely defined operators on
$\big(L^{2}(\R^{2n}), \|\boldsymbol\cdot\|_{L^{2}}\big)$.
As such, $b^{+}_{j}$ is the adjoint of $b_{j}$ and $\cL$ defines
as a densely defined self-adjoint operator on
$\big(L^{2}(\R^{2n}), \|\boldsymbol\cdot\|_{L^{2}}\big)$.
The following result was established in \cite[Theorem\,1.15]{MM08a}
(cf. also \cite[Theorem\,4.1.20]{MM07}).

\begin{thm}\label{lut1.1}
The spectrum of $\cL$ on $L^{2}(\R^{2n})$ is given by
\begin{align}\label{lus1.2a}
\textup{Spec}(\cL)=\big\{4\pi|\alpha|:  \alpha\in \N^{n}\big\},
\end{align}
and an orthogonal basis of the eigenspace of $4\pi|\alpha|$
is given by
\begin{align}\label{lus1.3a}
b^{\alpha}\big(z^{\beta}
\exp\big(-\pi\sum_{j}|z_{j}|^{2}/2\big)\big),\ \
\textup{with}\ \beta\in \N^{n}.
\end{align}
In particular, an orthonormal basis of $\Ker(\cL)$ is
\begin{align}\label{lus1.4a}
\bigg\{\phi_{\beta}(z)
=\Big(\frac{\pi^{|\beta|}}{\beta!}\Big)^{1/2}z^{\beta}
e^{-\pi\sum_{j}|z_{j}|^{2}/2}: \beta\in \N^{n}\bigg\}.
\end{align}
\end{thm}

Let $\cP(Z, Z')$ denote the kernel of the orthogonal projection
$\cP: L^{2}(\R^{2n})\rightarrow \Ker(\cL)$ with
respect to $dZ'$. Set $\cP^{\bot}=\textup{Id}-\cP$.

Obviously
$\cP(Z, Z')=\sum_{\beta}\phi_{\beta}(z)\ov{\phi_{\beta}(z')}$,
so we infer from (\ref{lus1.4a}) that
\begin{align}\label{lus1.5a}
\cP(Z, Z')=\exp\Big(-\frac{\pi}{2}\sum_{j=1}^{n}
\big(|z_{j}|^{2}+|z'_{j}|^{2}-2z_{j}\ov{z}'_{j}\big)\Big).
\end{align}
By (\ref{lus1.1a}) and (\ref{lus1.5a}), we obtain
\begin{align}\label{1.6}
\big(b^{+}_{j}\cP\big)(Z, Z')=0,\ \ \big(b_{j}\cP\big)(Z, Z')
=2\pi(\ov{z}_{j}-\ov{z}'_{j})\cP(Z, Z').
\end{align}
The following commutation relations are very useful in the
computations. Namely, for any polynomial $g(z, \ov{z})$ in
$z$ and $\ov{z}$, we have
\begin{equation}\label{1.6a}
\begin{split}
[b_{j}, b^{+}_{k}]&=
b_{j}b^{+}_{k}-b^{+}_{k}b_{j}=-4\pi\delta_{jk},\\
[b_{j}, b_{k}]&=[b^{+}_{j}, b^{+}_{k}]=0,\\
[g(z, \ov{z}), b_{j}]&=
2\frac{\p}{\p z_{j}}g(z, \ov{z}),  \\
[g(z, \ov{z}), b^{+}_{j}]&=
-2\frac{\p}{\p\ov{z}_{j}}g(z, \ov{z}).
\end{split}
\end{equation}
For a polynomial $F$ in $Z, Z'$, we denote by $F\cP$ the operator
on $L^{2}(\R^{2n})$ defined by  the kernel
$F(Z, Z')\cP(Z, Z')$ and the volume form $dZ$.

In the calculations involving the kernel
$\cP(\boldsymbol\cdot, \boldsymbol\cdot)$,
we prefer however to use the orthogonal decomposition of
$L^{2}(\R^{2n})$ given in Theorem \ref{lut1.1} and the fact that
$\cP$ is an orthogonal projection, rather than
integrating against the expression \eqref{lus1.5a} of
$\cP(\boldsymbol\cdot\,, \boldsymbol\cdot)$.
This point of view leads to streamlined computations
and to a better understanding of the operators involved. As an example,
Theorem \ref{lut1.1} implies that
\begin{eqnarray}\label{lus1.7a}
\big(\cP b^{\alpha}z^{\beta}\cP\big)(Z, Z')=
\begin{cases}
\big(z^{\beta}\cP\big)(Z, Z'),\  &\textup{if}\ |\alpha|=0, \\
0,\ \ &\textup{if}\ |\alpha|>0.
\end{cases}
\end{eqnarray}
We will also identify $z$ to
$\sum_{j}z_{j}\frac{\p}{\p z_{j}}$ and
$\ov{z}$ to
$\sum_{j}\ov{z}_{j}\frac{\p}{\p \ov{z}_{j}}$
when we consider $z$ and $\ov{z}$ as vector fields,
and
\begin{align}\label{lus1.8}
\mR=\sum_{j}Z_{j}\frac{\p}{\p Z_{j}}=z+\ov{z}=Z.
\end{align}

\section{Asymptotic expansion of the
generalized Bergman kernel}\label{s2}

Let $a^{X}$ be the injectivity radius
of $(X, g^{TX})$. We denote by
$B^{X}(x, \varepsilon)$ and $B^{T_{x}X}(0, \varepsilon)$
the open balls in $X$ and $T_{x}X$ with center $x$
and radius $\varepsilon$, respectively. Then the exponential map
$T_{x}X\ni Z\rightarrow \textup{exp}^{X}_{x}(Z)\in X$
is a diffeomorphism from $B^{T_{x}X}(0, \varepsilon)$ onto
$B^{X}(x, \varepsilon)$ for $\varepsilon\leqslant a^{X}$.
From now on, we identify $B^{T_{x}X}(0, \varepsilon)$ with
$B^{X}(x, \varepsilon)$ via the exponential map
for $\varepsilon\leqslant a^{X}$.

We fix $x_{0}\in X$.
For $Z\in B^{T_{x_{0}}X}$ we identify $(L_{Z}, h_{Z}^{L})$
 to $(L_{x_{0}}, h_{x_{0}}^{L})$ by parallel transport with respect
 to the connection $\nabla^{L}$ along the curve
$\gamma_{Z}: [0, 1]\ni u\rightarrow \textup{exp}^{X}_{x_{0}}(uZ)$.

In general, for functions in normal coordinates, we will
add a subscript $x_{0}$ to indicate the base point $x_{0}\in X$.
Similarly, $P_{\mH_{p}}(x, y)$ induces in terms of
the above trivialization (note that $\End({L^{p}_{x_{0}}})=\C$)
a smooth function
\[
\big\{(Z, Z')\in TX\times_{X}TX: |Z|, |Z'|<\varepsilon\big\}
\ni (Z, Z')\longmapsto P_{\mH_{p}, x_{0}}(Z, Z')\in \C\,,
\]
which also depends smoothly on the parameter $x_{0}$.

Let us choose an orthonormal
basis $\{w_{j}\}^{n}_{j=1}$ of $T_{x_{0}}^{(1, 0)}X$.
Then $e_{2j-1}=\frac{1}{\sqrt{2}}(w_{j}+\ov{w}_{j})$
and $e_{2j}=\frac{\sqrt{-1}}{\sqrt{2}}(w_{j}-\ov{w}_{j})$,
$j=1, \ldots, n$, forms an orthonormal basis of $T_{x_{0}}X$.
We use coordinates on $T_{x_{0}}X\simeq \R^{2n}$ given
by the identification
\begin{align}\label{2.1}
\R^{2n}\ni (Z_{1}, \ldots, Z_{2n}) \longmapsto
\sum^{2n}_{j=1}Z_{j}e_{j}\in T_{x_{0}}X.
\end{align}
In the sequel we also use complex coordinates
$z=(z_{1}, \ldots, z_{n})$ on $\C^{n}\simeq \R^{2n}$.

Let $dv_{TX}$ be the Riemannian
volume form on $(T_{x_{0}}X, g^{T_{x_{0}}X})$.
Let $\kappa_{x_{0}}: T_{x_{0}}X\rightarrow \R$,
$Z\mapsto \kappa_{x_{0}}(Z)$  be a smooth positive
function defined by
\begin{align}\label{2.2}
dv_{X}(Z)=\kappa_{x_{0}}(Z)dv_{TX}(Z),\ \ \kappa_{x_{0}}(0)=1,
\end{align}
where the subscript $x_{0}$ of $\kappa_{x_{0}}(Z)$ indicates
the base point $x_{0}\in X$.

{\bf Rescaling $\Delta_{p, \Phi}$ and Taylor expansion.}
For $s\in \cC^{\infty}(\R^{2n}, \C)$, $Z\in \R^{2n}$,
$|Z|\leqslant \varepsilon$, and for $t=\frac{1}{\sqrt{p}}$, set
\begin{align}\label{2.3}
\big(S_{t}s\big)(Z):= s(Z/t),\ \ \ \cL_{t}:=
S^{-1}_{t}\kappa^{1/2}t^{2}\Delta_{p, \Phi}\kappa^{-1/2}S_{t}.
\end{align}
For $U\in T_{x_{0}}X$, we denote $\nabla_{U}$ the ordinary
differential in direction $U$. Set
\begin{align}
\nabla_{0, \bullet}=\nabla_{\bullet}
+\frac{1}{2}R_{x_{0}}^{L}(Z, \boldsymbol\cdot),\ \
\cL_{0}=-\sum^{2n}_{j=1}(\nabla_{0, e_{j}})^{2}-2n\pi
=\sum^{n}_{j=1}b_{j}b^{+}_{j}=\cL.
\end{align}

\noindent
By \cite[Theorem\,1.4]{MM08a}, there exist second order
differential operators $\mO_{r}$ such that we
have an asymptotic expansion in $t$ when $t\rightarrow 0$,
\begin{align}\label{2.4}
\cL_{t}=\cL_{0}+\sum^{m}_{r=1}t^{r}\mO_{r}+\cO(t^{m+1}).
\end{align}
Moreover,
\begin{align}\label{2.5}
\mO_{1}(Z)=&-\frac{2}{3}\big(\partial_{j}R^{L}\big)_{x_{0}}
(\mR, e_{i})Z_{j}\nabla_{0, e_{i}}
-\frac{1}{3}\big(\partial_{i}R^{L}\big)_{x_{0}}(\mR, e_{i}),
\end{align}
and
\begin{equation}\label{2.6}
\begin{split}
\mO_{2}(Z)&=\frac{1}{3}
\Big\langle R^{TX}_{x_{0}}(\mR, e_{i})\mR, e_{j} \Big\rangle_{x_{0}}
\nabla_{0, e_{i}}\nabla_{0, e_{j}}\\
 &\qquad+\Big[\frac{2}{3}
 \Big\langle R_{x_{0}}^{TX}(\mR, e_{j})e_{j},
e_{i}\Big\rangle_{x_{0}}
-\frac{1}{2}\sum_{|\alpha|=2}(\partial^{\alpha}R^{L})_{x_{0}}
\frac{Z^{\alpha}}{\alpha!}(\mR, e_{i}) \Big]\nabla_{0, e_{i}}\\
 &\qquad-\frac{1}{4}\nabla_{e_{i}}\Big(\sum_{|\alpha|=2}
(\partial^{\alpha}R^{L})_{x_{0}}(\mR, e_{i})
\frac{Z^{\alpha}}{\alpha!}\Big)
-\frac{1}{9}\sum_{i}\Big[\sum_{j}(\partial_{j}R^{L})_{x_{0}}
(\mR, e_{i})Z_{j}\Big]^{2} \\
 &\qquad-\frac{1}{12}\Big[\cL_{0},
\Big\langle R_{x_{0}}^{TX}(\mR, e_{i})\mR, e_{i}\Big\rangle_{x_{0}}
\Big]+\Phi_{x_{0}}.
\end{split}
\end{equation}
From (\ref{2.1}) and (\ref{2.3}), as in \cite[Remark\,4.1.8]{MM07},
$\cL_{t}$ is a formally self-adjoint elliptic operator with
respect to $\|\boldsymbol\cdot\|_{L^{2}}$ on $\R^{2n}$ and
is a smooth
family of operators with respect to the parameter $x_{0}\in X$.
Thus $\cL, \cL_{0}$ and $\mO_{r}$ in (\ref{2.4}) are formally
self-adjoint with respect to $\|\boldsymbol\cdot\|_{L^{2}}$.

By \cite[Theorem 8.3.8]{MM07}, the following asymptotic expansion
of the generalized Bergman kernel holds.
\begin{thm} \label{t2.1}
There exist polynomials $J_{r}(Z, Z')$ in $Z, Z'$ with
the same parity as $r$ and of degree $\deg J_{r}(Z, Z')\leqslant 3r$,
such that if we define
\begin{align}\label{lus2.7}
\cF_{r}(Z, Z')=J_{r}(Z, Z')\cP(Z, Z'),\ \ J_{0}=1,
\end{align}
then for any $k, \ell,m\in \N$, $q>0$, there exists $C>0$ such that
if $p\geqslant 1$, $Z, Z'\in T_{x_{0}}X$ and
$|Z|, |Z'|\leqslant \frac{q}{\sqrt{p}}$, we have
 \begin{multline}\label{lus2.8}
\sup_{|\alpha|+|\alpha'|\leqslant m} \Big|
\frac{\partial^{|\alpha|+|\alpha'|}} {\partial Z^{\alpha}
{\partial Z'}^{\alpha'}}\Big ( \frac{1}{p^n}P_{\mH_p} (Z,Z')
 - \sum_{r=0}^k \cF_{r} (\sqrt{p}Z,\sqrt{p}Z')
\kappa ^{-\frac{1}{2}}(Z)\kappa^{-\frac{1}{2}}(Z')
p^{-\frac{r}{2}}\Big ) \Big |_{\cC ^{\ell}(X)}
 \leqslant C p^{-\frac{k-m+1}{2}},
\end{multline}
where $\cC ^{\ell}(X)$ is $\cC^\ell$-norm for
the parameter $x_0\in X$.
\end{thm}
\noindent
Moreover, by \cite[(4.1.93), (8.3.45)]{MM07}, $\cF_{1}$
and $\cF_{2}$ are given by
(cf. \cite[(8.3.65)]{MM07}, \cite[(1.111)]{MM08a})
\begin{equation}\label{lus2.9}\begin{split}
\cF_{1}&=-\cP^{\bot}\cL^{-1}\mO_{1}\cP
-\cP\mO_{1}\cL^{-1}\cP^{\bot},\\
\cF_{2}&=\cL^{-1}\cP^{\bot}\mO_{1}\cL^{-1}\cP^{\bot}\mO_{1}
\cP-\cL^{-1}\cP^{\bot}\mO_{2}\cP \\
&\quad +\cP\mO_{1}\cL^{-1}\cP^{\bot}\mO_{1}\cL^{-1}\cP^{\bot}
-\cP\mO_{2}\cL^{-1}\cP^{\bot}\\
&\quad +\cP^{\bot}\cL^{-1}\mO_{1}\cP\mO_{1}\cL^{-1}\cP^{\bot}
-\cP\mO_{1}\cL^{-2}\cP^{\bot}\mO_{1}\cP.
\end{split}\end{equation}
From Theorem \ref{t2.1}, we get in particular
\cite[Theorem\,8.3.3]{MM07}: there exist
$\bb_{r}\in \cC^{\infty}(X, \R)$ such that
for any $k, \ell\in \N$, there exists $C_{k, \ell}>0$ such that
\begin{align}\label{2.11}
\Big|\frac{1}{p^{n}}P_{\mH_{p}}(x, x)
-\sum^{k}_{r=0}\bb_{r}(x)p^{-r}\Big|_{\cC^{\ell}}
\leqslant C_{k, \ell}\ p^{-k-1},
\end{align}
and
\begin{align}\label{2.12}
\bb_{0}(x_{0})=\cF_{0}(0, 0)=1,\quad
\bb_{r}(x_{0})=\cF_{2r}(0, 0),\quad \cF_{2r+1}(0, 0)=0.
\end{align}

\section{Proof of Theorem \ref{t0.1}}\label{s3}
In this section we reduce Theorem \ref{t0.1} to Theorem \ref{t3.1}.
Let us fix $x_{0}\in X$. As in section \ref{s2},
we identify a small geodesic ball
$B^{X}(x_{0}, \varepsilon)$ to $B^{T_{x_{0}}X}$ by means
of the exponential map and we trivialize $L$ by using a unit frame
$e_{L}(Z)$ which is parallel with
respect to $\nabla^{L}$ along the curve
$[0, 1]\ni u\rightarrow uZ$ for
$Z\in B^{T_{x_{0}}X}(0, \varepsilon)$.

Set $d_{p}:=\dim\mH_{p}$ and for
$v=(v_{1}, \ldots, v_{d_{p}})\in \C^{d_{p}}$,
set $\|v\|^{2}=\sum^{d_{p}}_{j=1}|v_{j}|^{2}$.
We can now express the Fubini-Study form in the
homogeneous coordinate
$[v]=[v_{1}, \ldots, v_{d_{p}}]\in
\mathbb{P}(\mH^{\ast}_{p})$ as
\begin{align}\label{3.1}
\frac{\sqrt{-1}}{2\pi}\partial\ov{\p}
\log \big(\|v\|^{2}\big)=\frac{\sqrt{-1}}{2\pi}
\Biggl[\frac{1}{\|v\|^{2}}\sum^{d_{p}}_{j=1}
dv_{j}\wedge d\ov{v}_{j}
-\frac{1}{\|v\|^{4}}\sum^{d_{p}}_{j, k=1}
\ov{v}_{j}v_{k}dv_{j}\wedge d\ov{v}_{k}\Biggr].
\end{align}

Let $\{s_{j}\}$ be an orthonormal basis of $\mH_{p}$,
and let $\{s^{j}\}$ be its dual basis. We write locally
$s_{j}=f_{j}e^{\otimes p}_{L}$, then by (\ref{0.5}),
as in \cite[(5.1.17)]{MM07}, we have
\begin{align}\label{luw3.2}
\Phi_{p}(x)=\Biggl[\sum^{d_{p}}_{j=1}f_{j}(x)s^{j}\Biggr]\in
\mathbb{P}(\mH^{\ast}_{p}).
\end{align}
Set
\begin{align}\label{3.4}
f^{p}(x, y)=\sum_{i=1}^{d_{p}}f_{i}(x)\overline{f_{i}}(y)\ \
\textup{and}\ \ \big|f^{p}(x)\big|^{2}=f^{p}(x, x).
\end{align}
Then
\begin{align}\label{3.5}
P_{\mH_{p}}(x, y)=f^{p}(x, y)e^{\otimes p}_{L}(x)
\otimes e^{\otimes p}_{L}(y)^{\ast}, \quad \big|f^{p}(x)\big|^{2}
=P_{\mH_{p}}(x, x).
\end{align}
By (\ref{3.1}), (\ref{luw3.2}) and (\ref{3.4}), we get
\begin{equation}\label{lus3.2}
\begin{split}
\Phi_{p}^{\ast}(\omega_{FS})(x_{0})
&= \frac{\sqrt{-1}}{2\pi}\Biggl[\frac{1}{|f^{p}|^{2}}
\sum^{d_{p}}_{j=1}df_{j}\wedge d\ov{f_{j}}
-\frac{1}{|f^{p}|^{4}}\sum_{j, k=1}^{d_{p}}
\ov{f_{j}}f_{k}df_{j}\wedge d\ov{f_{k}}\Biggr](x_{0})\\
&=\frac{\sqrt{-1}}{2\pi}\Big[\big|f^{p}(x_{0})\big|^{-2}
d_{x}d_{y}f^{p}(x, y)
-\big|f^{p}(x_{0})\big|^{-4}d_{x}f^{p}(x, y) \wedge
d_{y}f^{p}(x, y)\Big]\Big|_{x\,=\,y\,=\,x_{0}}\:,
\end{split}
\end{equation}
where $\big|_{x\,=\,y\,=\,x_{0}}$ means the pull-back
by the diagonal map $\jmath: X\rightarrow X\times X$,
$x_{0}\mapsto (x_{0}, x_{0})$.

By (\ref{3.5}), $P_{\mH_{p}}(x, y)$ is represented by $f^{p}(x, y)$ under
our trivialization of $L$.
Since we work with normal coordinates, we get from \eqref{2.2}
(cf.\ \cite[(4.1.101)]{MM07})
\begin{equation}\label{bk2.82}
\begin{split}
&\kappa(Z)
= 1 + \cO (|Z|^2).
\end{split}
\end{equation}
By (\ref{lus2.8}), (\ref{2.12}), (\ref{lus3.2}) and (\ref{bk2.82}),
we get
\begin{equation}\label{lus3.6}
\begin{split}
\frac{1}{p}\Phi_{p}^{\ast}(\omega_{FS})(x_{0})
&=\frac{\sqrt{-1}}{2\pi}\bigg\{\Big[\frac{1}{\cF_{0}}
d_{x}d_{y}\cF_{0}-\frac{1}{\cF_{0}^{2}}d_{x}\cF_{0}
\wedge d_{y}\cF_{0}\Big](0, 0)\\
&\quad+p^{-1/2}\Big[\frac{1}{\cF_{0}}d_{x}d_{y}\cF_{1}
-\frac{1}{\cF^{2}_{0}}\big(d_{x}\cF_{1}\wedge d_{y}\cF_{0}
+d_{x}\cF_{0}\wedge d_{y}\cF_{1}\big)\Big](0, 0)\\
&\quad+p^{-1}\Big[\frac{1}{\cF_{0}}d_{x}d_{y}\cF_{2}
-\frac{\cF_{2}}{\cF^{2}_{0}} d_{x}d_{y}\cF_{0}
+\frac{2\cF_{2}}{\cF^{3}_{0}} d_{x}\cF_{0}\wedge d_{y}\cF_{0} \\
&\qquad\qquad
-\frac{1}{\cF^{2}_{0}}\big(d_{x}\cF_{0}\wedge d_{y}\cF_{2}
+d_{x}\cF_{1}\wedge d_{y}\cF_{1}
+d_{x}\cF_{2}\wedge d_{y}\cF_{0}\big)\Big](0, 0)\\
&\quad+p^{-3/2}\Big[\frac{1}{\cF_{0}}d_{x}d_{y}\cF_{3}
-\frac{\cF_{2}}{\cF_{0}^{2}} d_{x}d_{y}\cF_{1}
-\frac{1}{\cF^{2}_{0}}\big(d_{x}\cF_{0}\wedge d_{y}\cF_{3} \\
&\qquad\qquad
+d_{x}\cF_{1}\wedge d_{y}\cF_{2}
+d_{x}\cF_{2}\wedge d_{y}\cF_{1}
+d_{x}\cF_{3}\wedge d_{y}\cF_{0}\big) \\
 &\qquad\qquad
+\frac{2\cF_{2}}{\cF_{0}^{3}}
\big(d_{x}\cF_{0}\wedge d_{y}\cF_{1}
+d_{x}\cF_{1}\wedge d_{y}\cF_{0}\big)\Big](0, 0)\bigg\}
+\cO(p^{-2}).
\end{split}
\end{equation}
From (\ref{lus1.5a}) and (\ref{lus2.7}), we obtain
\begin{align}\label{lus3.7}
d_{x}\cF_{0}(0, 0)=d_{y}\cF_{0}(0, 0)=0.
\end{align}
As $J_{r}$ is a polynomial in $Z, Z'$ with the same parity as $r$,
we know from \eqref{lus1.5a} and \eqref{lus2.7}
that for $\alpha, \alpha'\in \N^{2n}$, there exists a polynomial
$J_{r, \alpha, \alpha'}$ in $Z, Z'$ with the same parity
as $r-|\alpha|-|\alpha'|$ such that
\begin{align}\label{luw3.8}
\frac{\p^{|\alpha|+|\alpha'|}}
{\p Z^{\alpha}\partial Z'^{\alpha'}}\cF_{r}(Z, Z')
=\big(J_{r, \alpha, \alpha'}\cP\big)(Z, Z').
\end{align}
In particular, \eqref{luw3.8} yields
\begin{align}\label{lus3.8}
d_{x}d_{y}\cF_{1}(0, 0)=d_{x}d_{y}\cF_{3}(0, 0)=0,
\quad d_{y}\cF_{2}(0, 0)=d_{x}\cF_{2}(0, 0)=0.
\end{align}
By (\ref{lus1.5a}) and (\ref{lus2.7}), we get
\begin{align}\label{3.9}
\frac{\sqrt{-1}}{2\pi}(d_{x}d_{y}\cF_{0})(0, 0)
=\frac{\sqrt{-1}}{2\pi}(d_{x}d_{y}\cP)(0, 0)
=\frac{\sqrt{-1}}{2}\sum^{n}_{j=1}dz_{j}\wedge d\ov{z}_{j}
=\omega(x_{0}).
\end{align}
Substituting (\ref{2.12}), (\ref{lus3.7}), (\ref{lus3.8})
and (\ref{3.9}) into (\ref{lus3.6}) yields
 \begin{equation}\label{3.10}
\begin{split}
 \frac{1}{p}\Phi_{p}^{\ast}(\omega_{FS})(x_{0})
 =\omega(x_{0})
 &+ \frac{\sqrt{-1}}{2\pi}\big(d_{x}d_{y}\cF_{2}
-d_{x}\cF_{1}\wedge d_{y}\cF_{1}\big)(0, 0) \, p^{-1}\\
&-\bb_{1}(x_{0})\omega(x_{0})\, p^{-1}
+\cO(p^{-2}).
\end{split}
 \end{equation}
Recall that for a tensor $\psi$, $\nabla^{X}\psi$ is the
covariant derivative of $\psi$ induced by the Levi-Civita connection
$\nabla^{TX}$. We will denote by
$\left\langle  \,\cdot ,\,\cdot\right\rangle$ the $\C$-bilinear form on
$TX\otimes_{\R} \C$ induced by $g^{TX}$.

The following observation \cite[(8.3.54)]{MM07} is very useful.
\begin{lemma}\label{4.9}
For $U\in T_{x_{0}}X$,
$\nabla_{U}^{X}J$ is skew-adjoint and the tensor
$\big\langle (\nabla_{\boldsymbol\cdot}^{X}J)\boldsymbol\cdot,
\boldsymbol\cdot \big\rangle$
is of type $\big(T^{\ast(1, 0)}X\big)^{\otimes 3}\oplus
\big(T^{\ast(0, 1)}X\big)^{\otimes 3}$.
\end{lemma}

 \begin{lemma}\label{t3.1a}
 We have
 \begin{align}\label{luw3.10}
 (d_{x}\cF_{1})(0, 0)=(d_{y}\cF_{1})(0, 0)=0.
 \end{align}
 \end{lemma}
 \begin{proof}
 By \eqref{lus1.1a} and \eqref{2.5}, we have
 (cf. \cite[(8.3.51)]{MM07})
 \begin{align}\label{lu0.10}
  \mO_{1}=
  -\frac{2}{3}\Big[\Big\langle \big(\nabla_{\mR}^{X}\mJ\big)\mR,
  \frac{\p}{\p z_{i}} \Big\rangle b^{+}_{i}-
b_{i}\Big\langle \big(\nabla_{\mR}^{X}\mJ\big)\mR,
\frac{\p}{\p \ov{z}_{i}} \Big\rangle \Big],\ \
\mJ=-2\pi\sqrt{-1}J.
\end{align}
From Theorem \ref{lut1.1}, \eqref{1.6}, \eqref{lu0.10}
and Lemma \ref{4.9}, we get (cf. \cite[(8.3.67)]{MM07})
\begin{equation}\label{lu0.7}
\begin{split}
\Big(\cL^{-1}&\cP^{\bot}\mO_{1}\cP\Big)(Z, Z') \\
&=  -\frac{\sqrt{-1}}{3}
\bigg\{\Big(\frac{b_{i}b_{j}}{4\pi}
\Big\langle \big(\nabla_{\frac{\p}{\p\ov{z}_{j}}}
^{X}J\big)\ov{z}',
\frac{\p}{\p \ov{z}_{i}}\Big\rangle
+b_{i}\Big\langle \big(\nabla_{\ov{z}'}^{X}J\big)\ov{z}',
\frac{\p}{\p \ov{z}_{i}}\Big\rangle  \Big)
\cP\bigg\}(Z, Z') \\
&= -\frac{\sqrt{-1}\pi}{3}
\bigg[\Big\langle \big(\nabla_{\ov{z}}^{X}J\big)\ov{z}',
\ov{z}\Big\rangle
+\Big\langle \big(\nabla_{\ov{z}'}^{X}J\big)\ov{z}',
\ov{z}\Big\rangle  \bigg]\cP(Z, Z').
\end{split}
\end{equation}
Note that if $K$ is an operator on
$\big(\R^{2n}, \|\boldsymbol\cdot\|_{L^{2}}\big)$
with smooth kernel
$K(Z, Z')$ with respect to $dZ'$, then the kernel $K^{\ast}(Z, Z')$
of the adjoint $K^{\ast}$ of $K$, with respect to $dZ'$, is given by
\begin{align}\label{luw3.14}
K^{\ast}(Z, Z')=\ov{K(Z', Z)}.
\end{align}
As $\cL$, $\mO_{1}$ are formally self-adjoint with respect to
$\|\boldsymbol\cdot\|_{L^{2}}$, thus
$\cP\mO_{1}\cL^{-1}\cP^{\bot}$ is the adjoint of
$\cL^{-1}\cP^{\bot}\mO_{1}\cP$.
From Lemma \ref{4.9}, (\ref{lu0.7}) and (\ref{luw3.14}), we get
\begin{equation}\label{lu0.9}
\begin{split}
\Big(\cP\mO_{1}\cL^{-1}\cP^{\bot}\Big)(Z, Z')
&= \frac{\sqrt{-1}\pi}{3}
\bigg[\Big\langle \big(\nabla_{z'}^{X}J\big)z, z'\Big\rangle
+\Big\langle \big(\nabla_{z}^{X}J\big)z, z'\Big\rangle  \bigg]
\cP(Z, Z')\\
&= \frac{\sqrt{-1}\pi}{3}
\Big\langle \big(\nabla_{\frac{\p}{\p z_{j}}}^{X}J
\big)\frac{\p}{\p z_{k}},
\frac{\p}{\p z_{l}}\Big\rangle \big(z_{j}'z_{k}z_{l}'
+z_{j}z_{k}z'_{l} \big)
\cP(Z, Z').
\end{split}
\end{equation}
As the coefficients of $\cP(Z, Z')$ in (\ref{lu0.7}) and (\ref{lu0.9})
are polynomials of degree $3$,
from (\ref{lus2.9}), (\ref{lu0.7}) and (\ref{lu0.9}),
we get (\ref{luw3.10}).
The proof of Lemma \ref{t3.1a} is completed.
 \end{proof}

 \begin{thm}\label{t3.1}
 The following identity holds,
 \begin{align}\label{lu1.1}
 \frac{\sqrt{-1}}{2\pi}(d_{x}d_{y}\cF_{2})(0, 0)
 =\bb_{1}(x_{0})\omega(x_{0}).
 \end{align}
 \end{thm}
Lemma \ref{t3.1a}, Theorem \ref{t3.1} and \eqref{3.10} yield
Theorem \ref{t0.1}.

\section{Proof of Theorem \ref{t3.1}}\label{s4}
This section is devoted to the proof of Theorem \ref{t3.1}.
We will compute the contribution of each term in
(\ref{lus2.9}) to $\cF_{2}$.
Set
\begin{align}\label{4.4}
I_{1}=&
\cL^{-1}\cP^{\bot}\mO_{1}\cL^{-1}\cP^{\bot}
\mO_{1}\cP, \quad I_{2}=-\cL^{-1}\cP^{\bot}\mO_{2}\cP,
\nonumber \\
I_{3}=&
\cP \mO_{1}\cL^{-1}\cP^{\bot}\mO_{1}\cL^{-1}\cP^{\bot},
\quad I_{4}=-\cP\mO_{2}\cL^{-1}\cP^{\bot},
 \\
I_{5}=&
\cP^{\bot}\cL^{-1}\mO_{1}\cP\mO_{1}\cL^{-1}\cP^{\bot},
 \quad I_{6}=-\cP\mO_{1}\cL^{-2}\cP^{\bot}\mO_{1}\cP.
\nonumber
\end{align}
For $j\in \{1, \ldots, 6\}$, let $I_{j}(Z, Z')$
be the smooth kernel of the operator $I_{j}$ with respect to $dZ'$.
By \eqref{lus2.9},
\begin{align}\label{lu0.5}
(d_{x}d_{y}\cF_{2})(0, 0)=\sum^{6}_{j=1}(d_{x}d_{y}I_{j})(0, 0).
\end{align}
In the context of (\ref{luw3.14}), by denoting
$b_{jk}=\dfrac{\p^{2}K}{\p Z_{j}\partial Z'_{k}}(Z, Z')
\Big|_{Z=Z'=0}$ we have
\begin{equation}\label{luw2.1}
\begin{split}
\big(d_{Z}d_{Z'}K^{\ast}\big)(0, 0)&=
\sum_{j, k}dZ_{j}\wedge dZ_{k}\ \frac{\p^{2}K^{\ast}}
{\p Z_{j}\partial Z'_{k}}(Z, Z')\big|_{Z=Z'=0} \\
&= \sum_{j<k}\big(\ov{b}_{kj}
-\ov{b}_{jk}\big)dZ_{j}\wedge dZ_{k}
 = -\ov{\big(d_{Z}d_{Z'}K\big)(0, 0)}.
\end{split}
\end{equation}
Since the operators $\mO_{r}$ from \eqref{2.4} are formally self-adjoint
with respect to $\|\boldsymbol\cdot\|_{L^{2}}$, \eqref{4.4} implies
that
$I_{1}$ and $I_{2}$ are the adjoints of $I_{3}$ and $I_{4}$, respectively,
as operators acting on
$\big(\R^{2n}, \|\boldsymbol\cdot\|_{L^{2}}\big)$.
Hence by \eqref{luw2.1},
\begin{align}\label{luw2.2}
(d_{x}d_{y}I_{3})(0, 0)=-\ov{(d_{x}d_{y}I_{1})(0, 0)},
\quad  (d_{x}d_{y}I_{4})(0, 0)=-\ov{(d_{x}d_{y}I_{2})(0, 0)}.
\end{align}

\subsection{Evaluation of $(d_{x}d_{y}I_{j})(0, 0)$
for $j=1, 3, 5, 6$}\label{luws4.1}

To simplify the notation, for polynomials $Q_{1}, Q_{2}$ in $Z, Z'$,
we will denote
\begin{align}\label{luw4.10}
(Q_{1}\cP)(Z, Z')\sim (Q_{2}\cP)(Z, Z'),
\end{align} if the constant coefficient and the coefficient of
$Z'_{j}$ for all $j$ in $Q_{1}-Q_{2}$
as a polynomial in $Z$ are zero; we denote
\begin{align}\label{luw4.11}
(Q_{1}\cP)(Z, Z')\approx (Q_{2}\cP)(Z, Z'),
\end{align}
if the constant coefficient
and the coefficients
of $Z_{j}, Z'_{k}, Z_{j}Z'_{k}$ for all $j, k$ in $Q_{1}-Q_{2}$ are zero.

Set
\begin{align}\label{luw4.13}
\cJ_{jir}:=\Big\langle
\big(\nabla_{\frac{\p}{\p z_{j}}}^{X}J\big)
\frac{\p}{\p z_{i}},
\frac{\p}{\p z_{r}}\Big\rangle, \quad
\cJ_{\ov{j}\,\ov{i}\ov{r}}:=
\Big\langle \big(\nabla_{\frac{\p}
{\p \ov{z}_{j}}}^{X}J\big)\frac{\p}{\p \ov{z}_{i}},
\frac{\p}{\p \ov{z}_{r}}\Big\rangle=\ov{\cJ_{jir}}.
\end{align}
From Lemma \ref{4.9}, (\ref{1.6}), (\ref{lu0.10}) and (\ref{lu0.7}), we get
\begin{equation}\label{lu0.28}
\begin{split}
\Big(&\mO_{1}\cL^{-1}\cP^{\bot}\mO_{1}\cP\Big)(Z, Z')\\
&=\frac{4\pi}{9}\bigg\{\Big(\Big\langle \big(\nabla_{z}^{X}J\big)z,
\frac{\p}{\p z_{i}} \Big\rangle b^{+}_{i}-
b_{i}\Big\langle \big(\nabla_{\ov{z}}^{X}J\big)\ov{z},
\frac{\p}{\p \ov{z}_{i}} \Big\rangle \Big)   \\
&\qquad\qquad\times\Big(\frac{b_{j}b_{k}}{4\pi}\Big\langle
\big(\nabla_{\frac{\p}{\p\ov{z}_{k}}}^{X}J\big)\ov{z}',
\frac{\p}{\p \ov{z}_{j}}\Big\rangle
+b_{j}\Big\langle \big(\nabla_{\ov{z}'}^{X}J\big)\ov{z}',
\frac{\p}{\p \ov{z}_{j}}\Big\rangle \Big)\cP\bigg\}(Z, Z') \\
&\sim \frac{1}{9}\bigg\{\Big[\Big\langle \big(\nabla_{z}^{X}J\big)z,
\frac{\p}{\p z_{i}} \Big\rangle b^{+}_{i}-
b_{i}\Big\langle \big(\nabla_{\ov{z}}^{X}J\big)\ov{z},
\frac{\p}{\p \ov{z}_{i}} \Big\rangle \Big]
b_{j}b_{k}\Big\langle
\big(\nabla_{\frac{\p}{\p\ov{z}_{k}}}^{X}J\big)\ov{z}',
\frac{\p}{\p \ov{z}_{j}}\Big\rangle\cP\bigg\}(Z, Z')\\
&\sim \frac{1}{9}\bigg\{\Big\langle \big(\nabla_{z}^{X}J\big)z,
\frac{\p}{\p z_{i}} \Big\rangle
\Big\langle \big(\nabla_{\frac{\p}
{\p\ov{z}_{k}}}^{X}J\big)\ov{z}',
\frac{\p}{\p \ov{z}_{j}}\Big\rangle b^{+}_{i}
b_{j}b_{k}\cP\bigg\}(Z, Z'),
\end{split}
\end{equation}
where in the last relation $\sim$ of \eqref{lu0.28}
we used
$b_{i}\Big\langle \big(\nabla_{\ov{z}}^{X}J\big)\ov{z},
\frac{\p}{\p \ov{z}_{i}} \Big\rangle \cP
= \Big\langle \big(\nabla_{\ov{z}}^{X}J\big)\ov{z},
-2\pi \ov{z}' \Big\rangle \cP$.

By Theorem \ref{lut1.1}, (\ref{1.6}) and  (\ref{1.6a}),
we get $b^{+}_{i}b_{j}b_{k}=b_{j}b_{k}b^{+}_{i}
+4\pi\big(\delta_{ij}b_{k}+\delta_{ik}b_{j}\big)$ and
\begin{align}\label{luw4.19}\begin{split}
\cL^{-1}\cP^{\bot}\big(z_{s}z_{t}b^{+}_{i}b_{j}b_{k}\cP\big)
&=4\pi\cL^{-1}\cP^{\bot}\big(z_{s}z_{t}(\delta_{ij}b_{k}
+\delta_{ik}b_{j})\cP\big)\\
&=4\pi\cL^{-1}\cP^{\bot}\big((\delta_{ij}b_{k}
 +\delta_{ik}b_{j})z_{s}z_{t}\cP\big) \\
 &= (\delta_{ij}b_{k}+\delta_{ik}b_{j})z_{s}z_{t}\cP.
\end{split}\end{align}
By \eqref{1.6}, \eqref{1.6a},
\eqref{4.4}, \eqref{lu0.28} and \eqref{luw4.19} we obtain
\begin{align}\label{lu0.31}\begin{split}
I_{1}(Z, Z') &\sim
\frac{1}{9}\cJ_{sti}\cJ_{\ov{k}\,\ov{l}\ov{j}}
\big((\delta_{ij}b_{k}z_{s}z_{t}
+\delta_{ik}b_{j}z_{s}z_{t})\ov{z}'_{l}\cP\big)(Z, Z')\\
&= \frac{1}{9}\cJ_{sti}\cJ_{\ov{k}\,\ov{l}\ov{j}}
\Big[-2\delta_{ij}\delta_{ks}z_{t}
-2\delta_{ij}\delta_{kt}z_{s}
+2\pi\delta_{ij}z_{s}z_{t}(\ov{z}_{k}-\ov{z}'_{k})\\
 &\qquad\qquad\qquad\qquad -2\delta_{ik}\delta_{js}z_{t}
 -2\delta_{ik}\delta_{jt}z_{s}
 +2\pi\delta_{ik}z_{s}z_{t}(\ov{z}_{j}
 -\ov{z}'_{j})\Big]\ov{z}'_{l}\cP(Z, Z').
\end{split}\end{align} Recall that $\cJ_{sti}$ is
anti-symmetric on $t$ and $i$, thus the contribution of
$-2\delta_{ij}\delta_{kt}z_{s}
-2\delta_{ik}\delta_{jt}z_{s}$ in (\ref{lu0.31}) is zero.
Thus \eqref{lu0.31} yields
\begin{align}\label{luw4.21}
I_{1}(Z, Z')\approx  -\frac{2}{9}\cJ_{sri}\cJ_{\ov{k}\,\ov{q}\ov{j}}
\big(\delta_{ik}\delta_{js}+\delta_{ij}\delta_{ks}\big)
z_{r}\ov{z}'_{q}\cP(Z, Z').
\end{align}
By Lemma \ref{4.9}, (\ref{lus1.5a}) and (\ref{luw4.21}), we get
{\allowdisplaybreaks
\begin{align}\label{lu0.32a}\begin{split}
(d_{x}d_{y}I_{1})(0, 0)=
-\frac{2}{9}\cJ_{jir}\big(\cJ_{\ov{i}\,\ov{j}\ov{q}}
+\cJ_{\ov{j}\,\ov{i}\ov{q}}\big)
dz_{r}\wedge d\ov{z}_{q}.
\end{split}\end{align}}
From (\ref{luw2.2}), (\ref{luw4.13}) and (\ref{lu0.32a}), we get
\begin{align}\label{lu0.35a}
(d_{x}d_{y}I_{3})(0, 0) = (d_{x}d_{y}I_{1})(0, 0).
\end{align}
By (\ref{lus1.5a}), (\ref{lu0.7}), (\ref{lu0.9}) and (\ref{4.4}),
we get
\begin{equation}\label{luw4.12}
\begin{split}
 I_{5}(Z, Z') &\sim
\frac{\pi^{2}}{9}\Bigg\{\bigg(\Big\langle
\big(\nabla_{\ov{z}}^{X}J\big)\ov{z}''
+ \big(\nabla_{\ov{z}''}^{X}J\big)\ov{z}'',
\ov{z}\Big\rangle \cP\bigg)
 \!\circ\! \bigg(\Big\langle \big(\nabla_{\frac{\p}{\p z_{j}}}^{X}J
\big)\frac{\p}{\p z_{k}}, \frac{\p}{\p z_{l}}
\Big\rangle
z''_{j}z''_{k}z'_{l}\cP\bigg)\Bigg\}(Z, Z') \\
&\approx
\frac{\pi^{2}}{9} \Big\langle \big(\nabla_{\frac{\p}
{\p \ov{z}_{s}}}^{X}J\big)\frac{\p}{\p \ov{z}_{t}},
\ov{z}\Big\rangle
\Big\langle \big(\nabla_{\frac{\p}{\p z_{j}}}^{X}J\big)
\frac{\p}{\p z_{k}}, z'\Big\rangle
\Big\{\cP\circ\big(\ov{z}''_{s}\ov{z}''_{t}z''_{j}z''_{k}
\cP\big)\Big\}(Z, Z') \\
 & \approx \frac{\pi^{2}}{9} \Big\langle \big(\nabla_{\frac{\p}
{\p \ov{z}_{s}}}^{X}J\big)\frac{\p}{\p \ov{z}_{t}},
\ov{z}\Big\rangle
\Big\langle \big(\nabla_{\frac{\p}{\p z_{j}}}^{X}J\big)
\frac{\p}{\p z_{k}}, z'\Big\rangle \cP(Z, Z')
\Big(\cP\circ \big(\ov{z}_{s}\ov{z}_{t}z_{j}z_{k}\cP\big)\Big)(0, 0),
\end{split}
\end{equation}
where in the last equation we use $\cP(0, 0)=1$, since we need to
compute the constant coefficient of $\cP$ in
$\cP\circ\big(\ov{z}''_{s}\ov{z}''_{t}z''_{j}z''_{k}\cP\big)$.

By (\ref{1.6}) and (\ref{1.6a}), we get
\begin{align}\label{lu0.16}\begin{split}
(\ov{z}_{s}\ov{z}_{t}z_{j}z_{k}\cP)(Z, 0)&=
\frac{1}{4\pi^{2}}(z_{j}z_{k}b_{s}b_{t}\cP)(Z, 0),
\\
z_{j}z_{k}b_{s}b_{t}&=
b_{s}b_{t}z_{j}z_{k}+2\delta_{js}b_{t}z_{k}
+2\delta_{jt}b_{s}z_{k}   \\
&\quad +2\delta_{ks}b_{t}z_{j}
+2\delta_{kt}b_{s}z_{j}+4\delta_{jt}\delta_{ks}
+4\delta_{js}\delta_{kt}.
\end{split}\end{align}
From Theorem \ref{lut1.1} and (\ref{lu0.16}), we get
\begin{align}\label{luw4.14}
\Big(\cP\circ\big(\ov{z}_{s}\ov{z}_{t}z_{j}z_{k}\cP\big)\Big)(0, 0)
&=\frac{1}{4\pi^{2}}(4\delta_{jt}\delta_{ks}
+4\delta_{js}\delta_{kt})\cP(0, 0)   \\
&=\frac{1}{\pi^{2}}(\delta_{jt}\delta_{ks}+\delta_{js}\delta_{kt}).
\nonumber
\end{align}
From (\ref{luw4.13}), (\ref{luw4.12}) and (\ref{luw4.14}), we obtain
\begin{align}\label{lu0.21a}
(d_{x}d_{y}I_{5})(0, 0)=
-\frac{1}{9}\cJ_{jir}\big(\cJ_{\ov{i}\,\ov{j}\ov{q}}
+\cJ_{\ov{j}\,\ov{i}\ov{q}}\big)
dz_{r}\wedge d\ov{z}_{q}.
\end{align}
By (\ref{lu0.7}), (\ref{lu0.9}) and (\ref{4.4}), we get
\begin{equation}\label{luw4.15}
\begin{split}
&I_{6}(Z, Z')\approx
-\frac{\pi^{2}}{9}\Big\langle
\big(\nabla_{\frac{\p}{\p z_{j}}}^{X}J\big)z,
\frac{\p}{\p z_{l}}\Big\rangle
\Big\langle \big(\nabla_{\frac{\p}
{\p \ov{z}_{s}}}^{X}J\big)\ov{z}', \frac{\p}
{\p \ov{z}_{t}}\Big\rangle
\Big(\cP\circ \big(z''_{j}z''_{l}\ov{z}''_{s}\ov{z}''_{t}
\cP\big)\Big)(Z, Z')  \\
 &\approx -\frac{\pi^{2}}{9}\Big\langle
\big(\nabla_{\frac{\p}{\p z_{j}}}^{X}J\big)z,
\frac{\p}{\p z_{k}}\Big\rangle
\Big\langle \big(\nabla_{\frac{\p}{\p \ov{z}_{s}}}^{X}J
\big)\ov{z}', \frac{\p}{\p \ov{z}_{t}}\Big\rangle \cP(Z, Z')
\Big(\cP\circ \big(z_{j}z_{k}\ov{z}_{s}\ov{z}_{t}
\cP\big)\Big)(0, 0).
\end{split}
\end{equation}
Thus by (\ref{luw4.13}), (\ref{luw4.14}) and (\ref{luw4.15}),
we get
{\allowdisplaybreaks
\begin{align}\label{lu0.28a}\begin{split}
(d_{x}d_{y}I_{6})(0, 0)=&
-\frac{1}{9}\cJ_{jrk}\cJ_{\ov{s}\,\ov{q}\ov{t}}
\big(\delta_{jt}\delta_{ks}+\delta_{js}\delta_{kt}\big)dz_{r}
\wedge d\ov{z}_{q}
 \\=&
-\frac{1}{9}\cJ_{jir}\big(\cJ_{\ov{i}\,\ov{j}\ov{q}}
+\cJ_{\ov{j}\,\ov{i}\ov{q}}\big)
dz_{r}\wedge d\ov{z}_{q}.
\end{split}\end{align}}

\subsection{Evaluation of $(d_{x}d_{y}I_{2})(0, 0)$: part I}
\label{luws4.2}
Recall that by \cite[Lemma\,2.1]{MM08a} we have
\begin{multline}\label{lu4.21}
\mO_{2}\cP=
\bigg\{\frac{1}{3}b_{i}b_{j}
\Big\langle R_{x_{0}}^{TX}(\mR, \frac{\p}{\p \ov{z}_{i}})
\mR, \frac{\p}{\p\ov{z}_{j}}\Big\rangle
+\frac{1}{2}b_{i}\Big[\sum_{|\alpha|=2}
\big(\partial^{\alpha}R^{L}\big)(\mR,
\frac{\p}{\p\ov{z}_{i}})\frac{Z^{\alpha}}{\alpha!}\Big]\\
+\frac{4}{3}b_{j}\Big[\Big\langle R^{TX}\Big(\frac{\p}{\p z_{i}},
\frac{\p}{\p\ov{z}_{i}}\Big)\mR,
\frac{\p}{\p\ov{z}_{j}}\Big\rangle
-\Big\langle R^{TX}\Big(\mR, \frac{\p}{\p z_{i}}\Big)
\frac{\p}{\p \ov{z}_{i}},
\frac{\p}{\p\ov{z}_{j}}\Big\rangle \Big] \\
-2\pi\sqrt{-1}\Big\langle
\big(\nabla^{X}\nabla^{X}J\big)_{(\mR, \mR)}
\frac{\p}{\p z_{i}}, \frac{\p}{\p \ov{z}_{i}}
\Big\rangle+
4\Big\langle R^{TX}\Big(\frac{\p}{\p z_{i}},
\frac{\p}{\p z_{j}}\Big)
\frac{\p}{\p\ov{z}_{i}},
\frac{\p}{\p \ov{z}_{j}}\Big\rangle\bigg\}\cP\\
+ \bigg[-\frac{1}{3}\cL\Big\langle R^{TX}\Big(\mR,
\frac{\p}{\p z_{j}}\Big)\mR,
\frac{\p}{\p\ov{z}_{j}}\Big\rangle
+\frac{4\pi^{2}}{9}\big|(\nabla_{\mR}^{X}J)\mR\big|^{2}
+\Phi_{x_{0}}\bigg]\cP.
\end{multline}
Set
{\allowdisplaybreaks
\begin{align}\label{4.21}\begin{split}
I_{21}(Z, Z')&=
\frac{1}{3}\bigg(\cL^{-1}\cP^{\bot}
b_{i}b_{j}\Big\langle R_{x_{0}}^{TX}(\mR,
\frac{\p}{\p \ov{z}_{i}})
\mR, \frac{\p}{\p\ov{z}_{j}}\Big\rangle \cP\bigg)(Z, Z'), \\
I_{22}(Z, Z')&=
\frac{1}{2}\bigg(\cL^{-1}\cP^{\bot}b_{i}\Big[\sum_{|\alpha|=2}
\big(\partial^{\alpha}R^{L}\big)\frac{Z^{\alpha}}{\alpha!}
(\mR, \frac{\p}{\p\ov{z}_{i}})\Big]\cP\bigg)(Z, Z'), \\
I_{23}(Z, Z')&=
\frac{4}{3}\bigg\{\cL^{-1}b_{j}
\Big\langle R^{TX}\Big(\frac{\p}{\p z_{i}},
\frac{\p}{\p\ov{z}_{i}}\Big)\mR
-  R^{TX}\Big(\mR, \frac{\p}{\p z_{i}}\Big) \frac{\p}{\p \ov{z}_{i}},
\frac{\p}{\p\ov{z}_{j}}\Big\rangle\cP\bigg\}(Z, Z'), \\
I_{24}(Z, Z')&=
-2\pi\sqrt{-1}\bigg(\cL^{-1}\cP^{\bot}
\Big\langle \big(\nabla^{X}\nabla^{X}J
\big)_{(\mR, \mR)}\frac{\p}{\p z_{i}},
\frac{\p}{\p \ov{z}_{i}} \Big\rangle \cP\bigg)(Z, Z'),\\
I_{25}(Z, Z')&=
-\frac{1}{3}\bigg(\cP^{\bot}\cL^{-1}\cL
\Big\langle R^{TX}\Big(\mR, \frac{\p}{\p z_{i}}\Big)\mR,
\frac{\p}{\p\ov{z}_{i}}\Big\rangle\cP\bigg)(Z, Z'), \\
I_{26}(Z, Z')&=
\frac{4\pi^{2}}{9}\Big(\cL^{-1}\cP^{\bot}
\big|(\nabla_{\mR}^{X}J)\mR\big|^{2}\cP\Big)(Z, Z').
\end{split}\end{align}}
By Theorem \ref{lut1.1},
$\cP^{\bot}\Big\langle R^{TX}\Big(\frac{\p}{\p z_{i}},
\frac{\p}{\p z_{j}}\Big)\frac{\p}{\p\ov{z}_{i}},
\frac{\p}{\p \ov{z}_{j}}\Big\rangle\cP
=\cP^{\bot}\Phi_{x_{0}}\cP=0$.
Thus \eqref{4.4}, \eqref{lu4.21} and \eqref{4.21} yield
\begin{align}\label{4.22}
-I_{2}(Z, Z')=\sum^{6}_{j=1}I_{2j}(Z, Z').
\end{align}
\comment{
As in \cite[(8.3.56), (8.3.63)]{MM07}, we have
\begin{align}\label{luw4.35}\begin{split}
& \big|\nabla^{X}J\big|^{2}=
\sum_{i, j}\big|(\nabla_{e_{i}}^{X}J)e_{j}\big|^{2}=
8\Big\langle \big(\nabla_{\frac{\p}{\p z_{i}}}^{X}J\big)
\frac{\p}{\p z_{j}},
\big(\nabla_{\frac{\p}{\p \ov{z}_{i}}}^{X}J\big)
\frac{\p}{\p \ov{z}_{j}}\Big\rangle,
\\ &
\Big\langle R^{TX}\Big(\frac{\p}{\p z_{i}},
\frac{\p}{\p z_{j}}\Big)\frac{\p}{\p \ov{z}_{i}},
\frac{\p}{\p\ov{z}_{j}}\Big\rangle
= \frac{1}{32}\big|\nabla^{X}J\big|^{2}.
\end{split}
\end{align}
}
We evaluate first the contribution of $I_{2j}$, $j=1,3,5,6$,
in $(d_{x}d_{y}I_{2})(0, 0)$. We recall the following
well-known symmetry properties of the
curvature $R^{TX}$:  for $U, V, W, Y\in TX$, we have
\begin{align}\label{luw4.36}\begin{split}
& \Big\langle R^{TX}(U, V)W, Y\Big\rangle
=\Big\langle R^{TX}(W, Y)U, V\Big\rangle,
\\ &
R^{TX}(U, V)W+R^{TX}(V, W)U+R^{TX}(W, U)V=0.
\end{split}
\end{align}
Using (\ref{1.6a}) and (\ref{luw4.36}), we have
\begin{multline}\label{4.55}
b_{i} b_{j}\Big\langle R^{TX}\Big(\mR,
\frac{\p}{\p \ov{z}_{i}}\Big)\mR,
\frac{\p}{\p \ov{z}_{j}}\Big\rangle
=b_{i} b_{j}\Big\langle R^{TX}\Big(\frac{\p}{\p z_{s}},
\frac{\p}{\p\ov{z}_{i}}\Big)\frac{\p}{\p z_{t}},
\frac{\p}{\p\ov{z}_{j}}\Big\rangle z_{s}z_{t}\\
+2 b_{i} b_{j}\Big\langle R^{TX}\Big(\frac{\p}{\p z_{s}},
\frac{\p}{\p\ov{z}_{i}}\Big)\frac{\p}{\p \ov{z}_{t}},
\frac{\p}{\p\ov{z}_{j}}\Big\rangle z_{s}\ov{z}_{t}
+b_{i} b_{j} \Big\langle R^{TX}\Big(\frac{\p}{\p \ov{z}_{s}},
\frac{\p}{\p\ov{z}_{i}}\Big)\frac{\p}{\p \ov{z}_{t}},
\frac{\p}{\p\ov{z}_{j}}\Big\rangle \ov{z}_{s}\ov{z}_{t}.
\end{multline}
By (\ref{1.6}) and (\ref{1.6a}), we get
\begin{align}\label{luw4.33}
z_{s}\ov{z}_{t}\cP(Z, Z')=
z_{s}(\frac{b_{t}}{2\pi}+\ov{z}'_{t})\cP(Z, Z')
 =\Big(\frac{b_{t}}{2\pi}z_{s}+\frac{\delta_{st}}{\pi}
+z_{s}\ov{z}'_{t}\Big)\cP(Z, Z').
\end{align}
By Theorem \ref{lut1.1}, (\ref{4.21}), (\ref{4.55})
and (\ref{luw4.33}), we get
{\allowdisplaybreaks
\begin{align}\label{lu0.50}\begin{split}
    3I_{21}(Z, Z')=&\:
\frac{1}{8\pi}\Big\langle R^{TX}\Big(\frac{\p}{\p z_{s}},
\frac{\p}{\p\ov{z}_{i}}\Big)\frac{\p}{\p z_{t}},
\frac{\p}{\p\ov{z}_{j}}\Big\rangle
\Big( b_{i}b_{j}z_{s}z_{t}\cP\Big)(Z, Z')\\
&+\Big\langle R^{TX}\Big(\frac{\p}{\p z_{s}},
\frac{\p}{\p\ov{z}_{i}}\Big) \frac{\p}{\p \ov{z}_{t}},
\frac{\p}{\p\ov{z}_{j}}\Big\rangle
\Big[\big(\frac{b_{i}b_{j}b_{t}}{12\pi^{2}}z_{s}
+\frac{b_{i}b_{j}}{4\pi}z_{s}\ov{z}'_{t}\big)\cP\Big](Z, Z')\\
&+\frac{1}{4\pi^{2}}\Big\langle R^{TX}\Big(\frac{\p}{\p z_{k}},
\frac{\p}{\p\ov{z}_{i}}\Big) \frac{\p}{\p \ov{z}_{k}},
\frac{\p}{\p\ov{z}_{j}}\Big\rangle
\Big(b_{i}b_{j}\cP\Big)(Z, Z')
+I_{27}(Z, Z'),
\end{split}\end{align}}
\noindent
where
\begin{align}
I_{27}(Z, Z')=\Big\langle R^{TX}\Big(\frac{\p}{\p \ov{z}_{s}},
\frac{\p}{\p\ov{z}_{i}}\Big)
\frac{\p}{\p \ov{z}_{t}},
\frac{\p}{\p\ov{z}_{j}}\Big\rangle
\Big(\cL^{-1}\cP^{\bot}b_{i}b_{j}\ov{z}_{s}\ov{z}_{t}
\cP\Big)(Z, Z').
\end{align}
Note that by Theorem \ref{lut1.1}, (\ref{1.6}) and (\ref{1.6a}),
\begin{equation}\label{luw4.50}
\begin{split}
4\pi^{2}\cL^{-1}\cP^{\bot}b_{i}b_{j}\ov{z}_{s}\ov{z}_{t}\cP
&= \cL^{-1}\cP^{\bot}b_{i}b_{j}\big(b_{s}+2\pi\ov{z}'_{s}\big)
(b_{t}+2\pi\ov{z}'_{t}\big)\cP \\
 &= \Big[\frac{1}{16\pi}b_{i}b_{j}b_{s}b_{t}
+\frac{1}{6}b_{i}b_{j}\big(b_{s}\ov{z}'_{t}+b_{t}\ov{z}'_{s}\big)
+\frac{\pi}{2}b_{i}b_{j}\ov{z}'_{s}\ov{z}'_{t}\Big]\cP.
\end{split}
\end{equation}
Thus, from (\ref{1.6}), (\ref{luw4.11}) and (\ref{luw4.50}), we get
\begin{align}\label{luw4.51}
I_{27}(Z, Z')\approx 0.
\end{align}
From (\ref{1.6}) and (\ref{1.6a}), we get
{\allowdisplaybreaks
\begin{align}\label{lu0.51}\begin{split}
\big(b_{i}b_{j}\cP\big)(Z, Z')&=4\pi^{2}\big(\ov{z}_{i}
-\ov{z}'_{i}\big)\big(\ov{z}_{j}-\ov{z}'_{j}\big)\cP(Z, Z'),
 \\
\big(b_{i}b_{j}z_{s}\ov{z}'_{t}\cP\big)(Z, Z')&=
\Big[-4\pi\delta_{js}\ov{z}'_{t}\big(\ov{z}_{i}-\ov{z}'_{i}\big)
-4\pi \delta_{is}\ov{z}'_{t}(\ov{z}_{j}-\ov{z}'_{j})
\\&\qquad\qquad\qquad
+4\pi^{2}z_{s}\ov{z}'_{t}\big(\ov{z}_{i}-\ov{z}'_{i}\big)
\big(\ov{z}_{j}-\ov{z}'_{j}\big)\Big]\cP(Z, Z'),
 \\
 \big( b_{i}b_{j}z_{s}z_{t}\cP\big)(Z, Z')&=
 \Big[4\delta_{it}\delta_{js}-4\pi\delta_{js}z_{t}
 \big(\ov{z}_{i}-\ov{z}'_{i}\big)
+4\delta_{jt}\delta_{is}-4\pi\delta_{jt}z_{s}
\big(\ov{z}_{i}-\ov{z}'_{i}\big)
 \\ &\qquad\qquad\qquad
-4\pi\delta_{is}z_{t}\big(\ov{z}_{j}-\ov{z}'_{j}\big)
-4\pi\delta_{it}z_{s}\big(\ov{z}_{j}-\ov{z}'_{j}\big)
 \\&\qquad\qquad\qquad
  +4\pi^{2}z_{s}z_{t}\big(\ov{z}_{j}-\ov{z}'_{j}\big)
\big(\ov{z}_{i}-\ov{z}'_{i}\big)\Big]\cP(Z, Z'),
\end{split}\end{align}}
and
\begin{align}\label{lu0.52}
\Big(b_{i}b_{j}b_{t}z_{s}\cP\Big)(Z, Z')
=\Big[\big(-2\delta_{ts}b_{i}b_{j}-2\delta_{js}b_{i}b_{t}
-2\delta_{is}b_{j}b_{t}+z_{s}b_{i}b_{j}b_{t}\big)\cP\Big](Z, Z').
\end{align}
By (\ref{1.6}), (\ref{lu0.51}) and (\ref{lu0.52}), we get
\begin{align}\label{luw4.52}
\big(d_{x}d_{y}(b_{i}b_{j}\cP)\big)(0, 0)=0, \quad
\big(d_{x}d_{y}(b_{i}b_{j}b_{t}z_{s}\cP)\big)(0, 0)=0.
\end{align}
\noindent
Substituting (\ref{3.9}), (\ref{luw4.36}),
(\ref{luw4.51})--(\ref{luw4.52}) into (\ref{lu0.50}),
we obtain
\begin{align}\label{lu0.62a}\begin{split}
(d_{x}d_{y}I_{21})(0, 0)  =&
-\frac{\sqrt{-1}}{3}\Big\langle 2 R^{TX}\Big(\frac{\p}{\p z_{j}},
\frac{\p}{\p \ov{z}_{i}}\Big)\frac{\p}{\p z_{i}}
- R^{TX}\Big(\frac{\p}{\p z_{j}}, \frac{\p}{\p z_{i}}\Big)
\frac{\p}{\p \ov{z}_{i}}, \frac{\p}{\p \ov{z}_{j}}\Big\rangle
\omega(x_{0}) \\
 &+\frac{1}{3}\Big\langle 2 R^{TX}\Big(\frac{\p}{\p z_{r}},
\frac{\p}{\p \ov{z}_{j}}\Big)  \frac{\p}{\p z_{j}}
+ R^{TX}\Big(\frac{\p}{\p z_{j}}, \frac{\p}{\p z_{r}}\Big)
\frac{\p}{\p \ov{z}_{j}}, \frac{\p}{\p \ov{z}_{q}}\Big\rangle
dz_{r}\wedge d\ov{z}_{q} \\
&+ \frac{1}{3}\Big\langle R^{TX}\Big(\frac{\p}{\p z_{j}},
\frac{\p}{\p \ov{z}_{r}}\Big)
\frac{\p}{\p \ov{z}_{j}}
+ R^{TX}\Big(\frac{\p}{\p z_{j}}, \frac{\p}{\p \ov{z}_{j}}\Big)
\frac{\p}{\p \ov{z}_{r}}, \frac{\p}{\p \ov{z}_{q}}\Big\rangle
d\ov{z}_{r}\wedge d\ov{z}_{q}.
\end{split}\end{align}
By (\ref{4.21}),
\begin{multline}\label{lu0.54}
\frac{3}{4}I_{23}=
\Big\langle R^{TX}\Big(\frac{\p}{\p z_{i}},
\frac{\p}{\p\ov{z}_{i}}\Big)\frac{\p}{\p z_{s}}
- R^{TX}\Big(\frac{\p}{\p z_{s}},
\frac{\p}{\p z_{i}}\Big)\frac{\p}{\p\ov{z}_{i}},
\frac{\p}{\p \ov{z}_{j}}\Big\rangle
\cL^{-1}b_{j}z_{s}\cP\\
+ \Big\langle R^{TX}\Big(\frac{\p}{\p z_{i}},
\frac{\p}{\p\ov{z}_{i}}\Big)\frac{\p}{\p \ov{z}_{s}}
- R^{TX}\Big(\frac{\p}{\p \ov{z}_{s}},
\frac{\p}{\p z_{i}}\Big)\frac{\p}{\p\ov{z}_{i}},
\frac{\p}{\p \ov{z}_{j}}\Big\rangle
\cL^{-1}b_{j}\ov{z}_{s}\cP.
\end{multline}
By Theorem \ref{lut1.1}, (\ref{1.6}) and (\ref{1.6a}),
\begin{align}\label{lu0.55}\begin{split}
\Big(\cL^{-1}\cP^{\bot}b_{j}z_{s}\cP\Big)(Z, Z')&=
\frac{1}{4\pi}\Big(b_{j}z_{s}\cP\Big)(Z, Z')
\\&=
\frac{1}{4\pi}\Big(-2\delta_{js}+2\pi z_{s}\big(\ov{z}_{j}
-\ov{z}'_{j}\big)\Big)\cP(Z, Z').
\end{split} \end{align}
Note that by (\ref{1.6}),
$\ov{z}_{s}\cP=\Big(\dfrac{b_{s}}{2\pi}+\ov{z}'_{s}\Big)\cP$.
Thus from Theorem \ref{lut1.1}, we get
\begin{align}\label{lu0.56}\begin{split}
\Big(\cL^{-1}b_{j}\ov{z}_{s}\cP\Big)(Z, Z')
&=
\Big[\big(\frac{b_{j}b_{s}}{16\pi^{2}}
+\frac{b_{j}}{4\pi}\ov{z}'_{s}\big)\cP\Big](Z, Z')
\\ &=
\Big[\frac{1}{4}\big(\ov{z}_{j}-\ov{z}'_{j}\big)
\big(\ov{z}_{s}-\ov{z}'_{s}\big)
+\frac{1}{2}\ov{z}'_{s}\big(\ov{z}_{j}-\ov{z}'_{j}\big)\Big]
\cP(Z, Z').
\end{split}\end{align}
As in (\ref{luw4.52}), we get
\begin{align}\label{luw4.54}
\big(d_{x}d_{y}(\cL^{-1}b_{j}\ov{z}_{s}\cP)\big)(0, 0)
=\frac{1}{2}d\ov{z}_{j}\wedge d\ov{z}_{s}.
\end{align}
From (\ref{3.9}), (\ref{luw4.36}), (\ref{lu0.54}),
(\ref{lu0.55}) and (\ref{luw4.54}), we get
\begin{align}\label{lu0.66a}\begin{split}
(d_{x}d_{y}I_{23})(0, 0)
=&\:
\frac{4\sqrt{-1}}{3}\Big\langle R^{TX}
(\frac{\p}{\p z_{j}}, \frac{\p}{\p \ov{z}_{i}})
\frac{\p}{\p z_{i}}  -
 2 R^{TX}\Big(\frac{\p}{\p z_{j}}, \frac{\p}{\p z_{i}}\Big)
\frac{\p}{\p \ov{z}_{i}}, \frac{\p}{\p \ov{z}_{j}}\Big\rangle
\, \omega(x_{0}) \\&
-\frac{2}{3}\Big\langle R^{TX}\Big(\frac{\p}{\p z_{r}},
\frac{\p}{\p \ov{z}_{i}}\Big)\frac{\p}{\p z_{i}}
+2 R^{TX}\Big(\frac{\p}{\p z_{i}}, \frac{\p}{\p z_{r}}\Big)
\frac{\p}{\p \ov{z}_{i}}, \frac{\p}{\p \ov{z}_{q}}\Big\rangle
dz_{r}\wedge d\ov{z}_{q}
 \\&
-\frac{2}{3}\Big\langle R^{TX}\Big(\frac{\p}{\p z_{i}},
\frac{\p}{\p \ov{z}_{i}}\Big)
\frac{\p}{\p \ov{z}_{r}}
+R^{TX}\Big(\frac{\p}{\p z_{i}}, \frac{\p}{\p \ov{z}_{r}}\Big)
\frac{\p}{\p \ov{z}_{i}}, \frac{\p}{\p \ov{z}_{q}}\Big\rangle
 d\ov{z}_{r}\wedge d\ov{z}_{q}.
\end{split}\end{align}
Clearly, by Theorem \ref{lut1.1} and (\ref{4.21}),
\begin{multline}\label{lu0.40}
- 3 I_{25}(Z, Z')  =
\bigg(\cP^{\bot}\Big\langle R^{TX}\Big(\mR,
\frac{\p}{\p z_{i}}\Big)\mR,
\frac{\p}{\p\ov{z}_{i}}\Big\rangle\cP\bigg)(Z, Z') \\
 =\Big\langle R^{TX}\Big(\frac{\p}{\p z_{j}},
\frac{\p}{\p z_{i}}\Big)\frac{\p}{\p \ov{z}_{k}}
+R^{TX}\Big(\frac{\p}{\p \ov{z}_{k}},
\frac{\p}{\p z_{i}}\Big)\frac{\p}{\p z_{j}},
\frac{\p}{\p \ov{z}_{i}}\Big\rangle
\Big(\cP^{\bot}\circ\big(z_{j}\ov{z}_{k}\cP\big)\Big)(Z, Z') \\
+\Big\langle R^{TX}\Big(\frac{\p}{\p \ov{z}_{j}},
\frac{\p}{\p z_{i}}\Big)\frac{\p}{\p \ov{z}_{k}},
\frac{\p}{\p \ov{z}_{i}}\Big\rangle
\Big(\cP^{\bot}\circ\big(\ov{z}_{j}\ov{z}_{k}\cP\big)\Big)(Z, Z').
\end{multline}
\comment{
By (\ref{1.6}) and (\ref{1.6a}), we get
\begin{multline}\label{luw4.33}
z_{j}\ov{z}_{k}\cP(Z, Z')=
z_{j}(\frac{b_{k}}{2\pi}+\ov{z}'_{k})\cP(Z, Z') =
\Big(\frac{b_{k}}{2\pi}z_{j}+\frac{\delta_{jk}}{\pi}
+z_{j}\ov{z}'_{k}\Big)\cP(Z, Z').
\end{multline}
}
From Theorem \ref{lut1.1}, (\ref{1.6}) and (\ref{luw4.33}), we get
\begin{multline}\label{lu0.40a}
\Big(\cP^{\bot}\circ\big(z_{j}\ov{z}_{k}\cP\big)\Big)(Z, Z')
=\frac{1}{2\pi}\big(b_{k}z_{j}\cP\big)(Z, Z')
=\Big(-\frac{1}{\pi}\delta_{jk}+z_{j}(\ov{z}_{k}
-\ov{z}'_{k})\Big)\cP(Z, Z'),
\end{multline}
and
\begin{equation}\label{lu0.40b}
\begin{split}
\Big(\cP^{\bot}\circ\big(\ov{z}_{j}\ov{z}_{k}\cP\big)\Big)(Z, Z')
&= \bigg\{\cP^{\bot}\Big[\frac{1}{4\pi^{2}}b_{j}b_{k}
+\frac{1}{2\pi}(b_{j}\ov{z}'_{k}
+b_{k}\ov{z}'_{j})\Big]\cP\bigg\}(Z, Z') \\
&=\Big((\ov{z}_{j}-\ov{z}'_{j})(\ov{z}_{k}-\ov{z}'_{k})
+\ov{z}'_{k}(\ov{z}_{j}-\ov{z}'_{j})
+\ov{z}'_{j}(\ov{z}_{k}-\ov{z}'_{k})\Big)\cP(Z, Z')\\
&=\big(\ov{z}_{j}\ov{z}_{k}-\ov{z}'_{j}\ov{z}'_{k}\big)\cP(Z, Z').
\end{split}
\end{equation}
As in \cite[(8.3.56), (8.3.63)]{MM07}, we have
\begin{equation}\label{luw4.35}\begin{split}
& \big|\nabla^{X}J\big|^{2}=
\sum_{i, j}\big|(\nabla_{e_{i}}^{X}J)e_{j}\big|^{2}=
8\Big\langle \big(\nabla_{\frac{\p}{\p z_{i}}}^{X}J\big)
\frac{\p}{\p z_{j}},
\big(\nabla_{\frac{\p}{\p \ov{z}_{i}}}^{X}J\big)
\frac{\p}{\p \ov{z}_{j}}\Big\rangle,
\\ &
\Big\langle R^{TX}\Big(\frac{\p}{\p z_{i}},
\frac{\p}{\p z_{j}}\Big)\frac{\p}{\p \ov{z}_{i}},
\frac{\p}{\p\ov{z}_{j}}\Big\rangle
= \frac{1}{32}\big|\nabla^{X}J\big|^{2}.
\end{split}
\end{equation}
By (\ref{3.9}), (\ref{luw4.36}), (\ref{lu0.40}), (\ref{lu0.40a}),
(\ref{lu0.40b}) and (\ref{luw4.35}), we get
\begin{multline}\label{lu0.41a}
(d_{x}d_{y}I_{25})(0, 0)
=  -\frac{1}{3}\Big\langle R^{TX}\Big(\frac{\p}{\p z_{j}},
\frac{\p}{\p z_{i}}\Big)\frac{\p}{\p \ov{z}_{k}}
+R^{TX}\Big(\frac{\p}{\p \ov{z}_{k}},
\frac{\p}{\p z_{i}}\Big)\frac{\p}{\p z_{j}},
\frac{\p}{\p \ov{z}_{i}}\Big\rangle\\
\times \Big(2\sqrt{-1}\delta_{jk}\, \omega(x_{0})
-dz_{j}\wedge d\ov{z}_{k}\Big)   \\
=\Big[-\frac{1}{96}\big|\nabla^{X}J\big|^{2}
+\frac{1}{3}\Big\langle R^{TX}\Big(\frac{\p}{\p z_{i}},
\frac{\p}{\p \ov{z}_{j}}\Big)\frac{\p}{\p z_{j}},
\frac{\p}{\p \ov{z}_{i}}\Big\rangle\Big]
2\sqrt{-1}\omega(x_{0})   \\
+\frac{1}{3}\Big\langle R^{TX}\Big(\frac{\p}{\p z_{i}},
\frac{\p}{\p z_{r}}\Big)\frac{\p}{\p \ov{z}_{i}}
- R^{TX}\Big(\frac{\p}{\p z_{r}},
\frac{\p}{\p \ov{z}_{i}}\Big)\frac{\p}{\p z_{i}},
\frac{\p}{\p \ov{z}_{q}}\Big\rangle
dz_{r}\wedge d\ov{z}_{q}.
\end{multline}
Then by Lemma \ref{4.9} and (\ref{4.21}), we get
\begin{align}\label{lu0.43}
9I_{26}(Z, Z')=
8\pi^{2}\Big\langle
\big(\nabla_{\frac{\p}{\p z_{i}}}^{X}J\big)
\frac{\p}{\p z_{j}},
\big(\nabla_{\frac{\p}{\p \ov{z}_{s}}}^{X}J\big)
\frac{\p}{\p \ov{z}_{t}}\Big\rangle
\Big(\cL^{-1}\cP^{\bot}z_{i}z_{j}\ov{z}_{s}\ov{z}_{t}\cP\Big)(Z, Z').
\end{align}
By Theorem \ref{lut1.1}, (\ref{1.6}), (\ref{1.6a}) and (\ref{lu0.16}),
we get
\begin{equation}\label{luw4.37a}
\begin{split}
\Big(\cL^{-1}\cP^{\bot}z_{i}z_{j}b_{s}\cP\Big)(Z, Z')&=
\frac{b_{s}}{4\pi}z_{i}z_{j}\cP(Z, Z')
= \frac{1}{4\pi}\big(-2\delta_{is}z_{j}-2\delta_{js}z_{i}
+z_{i}z_{j}b_{s}\big)\cP(Z, Z'),
\end{split}
\end{equation}
\begin{equation}\label{luw4.37b}
\begin{split}
\Big(&\cL^{-1}\cP^{\bot}z_{i}z_{j}b_{s}b_{t}\cP\Big)(Z, Z')
=\frac{1}{2\pi}\Big(\frac{b_{s}b_{t}}{4}z_{i}z_{j}
+\delta_{it}b_{s}z_{j}+\delta_{jt}b_{s}z_{i}
+\delta_{is}b_{t}z_{j}+\delta_{js}b_{t}z_{i}\Big)\cP(Z, Z')\\
&=\frac{1}{2\pi}\Big(-3\delta_{js}\delta_{it}-3\delta_{jt}\delta_{is}
+\frac{1}{2}\delta_{it}z_{j}b_{s}+\frac{1}{2}\delta_{jt}z_{i}b_{s}+
\frac{1}{2}\delta_{is}z_{j}b_{t}+\frac{1}{2}\delta_{js}z_{i}b_{t}
+\frac{1}{4}z_{i}z_{j}b_{s}b_{t}\Big)\cP(Z, Z').
\end{split}
\end{equation}
By (\ref{1.6}), (\ref{luw4.37a})
and (\ref{luw4.37b}), we get
\begin{equation}\label{4.42}
\begin{split}
\Big(&\cL^{-1}\cP^{\bot}z_{i}z_{j}\ov{z}_{s}\ov{z}_{t}
\cP\Big)(Z, Z')  =
\Big(\cL^{-1}\cP^{\bot}z_{i}z_{j}\big(\frac{b_{s}}{2\pi}
+\ov{z}'_{s}\big)
\big(\frac{b_{t}}{2\pi}+\ov{z}'_{t}\big)\cP\Big)(Z, Z') \\
&= \frac{1}{4\pi^{2}}\Big(\cL^{-1}\cP^{\bot}z_{i}z_{j}b_{s}b_{t}
\cP\Big)(Z, Z')+\frac{1}{2\pi}\Big(\cL^{-1}\cP^{\bot}
\big(z_{i}z_{j}\ov{z}'_{t}b_{s}+z_{i}z_{j}\ov{z}'_{s}b_{t}\big)
\cP\Big)(Z, Z') \\
&= \frac{1}{4\pi^{2}}\Big\{-\frac{3}{2\pi}\delta_{it}\delta_{js}
-\frac{3}{2\pi}\delta_{jt}\delta_{is}+
\frac{1}{2}\delta_{it}z_{j}\big(\ov{z}_{s}-\ov{z}'_{s}\big)
+\frac{1}{2}\delta_{jt}z_{i}\big(\ov{z}_{s}-\ov{z}'_{s}\big)
+ \frac{1}{2}\delta_{is}z_{j}\big(\ov{z}_{t}-\ov{z}'_{t}\big)\\
&\qquad\qquad\qquad
+\frac{1}{2}\delta_{js}z_{i}\big(\ov{z}_{t}-\ov{z}'_{t}\big)
+\frac{\pi}{2}z_{i}z_{j}\big(\ov{z}_{t}-\ov{z}'_{t}\big)
\big(\ov{z}_{s}-\ov{z}'_{s}\big) \Big\}\cP(Z, Z') \\
&\qquad+ \frac{1}{8\pi^{2}}\Big\{\big[-2\delta_{is}z_{j}-2\delta_{js}z_{i}
+2\pi z_{i}z_{j}(\ov{z}_{s}-\ov{z}'_{s})\big]\ov{z}'_{t}\\
&\qquad\qquad\qquad+\big[-2\delta_{it}z_{j}-2\delta_{jt}z_{i}
+2\pi z_{i}z_{j}(\ov{z}_{t}-\ov{z}'_{t})\big]
\ov{z}'_{s}\Big\}\cP(Z, Z').
\end{split}
\end{equation}
By (\ref{3.9}), (\ref{lu0.43}) and (\ref{4.42}), we get
\begin{equation}\label{4.43}
\begin{split}
 9(d_{x}d_{y}I_{26})(0, 0)
 =& -\Big\langle \big(\nabla_{\frac{\p}{\p z_{i}}}^{X}J\big)
\frac{\p}{\p z_{j}}, \big(\nabla_{\frac{\p}{\p \ov{z}_{s}}}^{X}J\big)
\frac{\p}{\p \ov{z}_{t}}\Big\rangle \Big[3(\delta_{it}\delta_{js}
+\delta_{jt}\delta_{is})
(-2\sqrt{-1})\omega(x_{0})\\
&+ \delta_{it}dz_{j}\wedge d\ov{z}_{s}
+\delta_{jt}dz_{i}\wedge d\ov{z}_{s}
+\delta_{is}dz_{j}\wedge d\ov{z}_{t}
+\delta_{js}dz_{i}\wedge d\ov{z}_{t} \\
&+ 2\delta_{is}dz_{j}\wedge d\ov{z}_{t}
+2\delta_{js}dz_{i}\wedge d\ov{z}_{t}
+2\delta_{it}dz_{j}\wedge d\ov{z}_{s}
+2\delta_{jt}dz_{i}\wedge d\ov{z}_{s}\Big].
\end{split}
\end{equation}
By (\ref{4.43}),
\begin{multline}\label{lu0.46}
    (d_{x}d_{y}I_{26})(0, 0) =
\frac{2}{3}\sqrt{-1}\Big\langle
\big(\nabla_{\frac{\p}{\p z_{i}}}^{X}J\big)
\frac{\p}{\p z_{j}},
\big(\nabla_{\frac{\p}{\p \ov{z}_{i}}}^{X}J\big)
\frac{\p}{\p \ov{z}_{j}}
+\big(\nabla_{\frac{\p}{\p \ov{z}_{j}}}^{X}J\big)
\frac{\p}{\p \ov{z}_{i}}\Big\rangle \omega(x_{0})\\
-\frac{1}{3}\bigg[\Big\langle
\big(\nabla_{\frac{\p}{\p z_{i}}}^{X}J\big)
\frac{\p}{\p z_{r}},
\big(\nabla_{\frac{\p}{\p \ov{z}_{i}}}^{X}J\big)
\frac{\p}{\p \ov{z}_{q}}
+\big(\nabla_{\frac{\p}{\p \ov{z}_{q}}}^{X}J\big)
\frac{\p}{\p \ov{z}_{i}}\Big\rangle\\
+\Big\langle \big(\nabla_{\frac{\p}{\p z_{r}}}^{X}J\big)
\frac{\p}{\p z_{i}},
\big(\nabla_{\frac{\p}{\p \ov{z}_{i}}}^{X}J\big)
\frac{\p}{\p \ov{z}_{q}}
+\big(\nabla_{\frac{\p}{\p \ov{z}_{q}}}^{X}J\big)
\frac{\p}{\p \ov{z}_{i}}\Big\rangle\bigg]
dz_{r}\wedge d\ov{z}_{q}.
\end{multline}
Note that for $U, V, W\in TX$,
$\big\langle JU, V\big\rangle =\omega(U, V)$,
thus (cf. \cite[(8.3.48)]{MM07}),
\begin{align}\label{luw4.40}
\big\langle (\nabla_{U}^{X}J)V, W\big\rangle
+\big\langle (\nabla_{V}^{X}J)W, U\big\rangle
+\big\langle(\nabla_{W}^{X}J)U, V\big\rangle
=d\omega(U, V, W)=0.
\end{align}
From Lemma \ref{4.9}, (\ref{luw4.40}) and
$\big|\frac{\p}{\p z_{j}}\big|^{2}=\frac{1}{2}$, we have
\begin{align}\label{lu0.47}\begin{split}
\Big\langle \big(\nabla_{\frac{\p}{\p z_{i}}}^{X}J\big)
\frac{\p}{\p z_{r}},
\big(\nabla_{\frac{\p}{\p \ov{z}_{q}}}^{X}J\big)
\frac{\p}{\p\ov{z}_{i}}\Big\rangle
=
2\Big\langle \big(\nabla_{\frac{\p}{\p z_{i}}}^{X}J\big)
\frac{\p}{\p z_{j}}, \frac{\p}{\p z_{r}}
\Big\rangle
\bigg[\Big\langle
\big(\nabla_{\frac{\p}{\p \ov{z}_{i}}}^{X}J\big)
\frac{\p}{\p \ov{z}_{j}},
\frac{\p}{\p\ov{z}_{q}}\Big\rangle
- \Big\langle \big(\nabla_{\frac{\p}{\p \ov{z}_{j}}}^{X}J
\big)\frac{\p}{\p \ov{z}_{i}},
\frac{\p}{\p\ov{z}_{q}}\Big\rangle
\bigg].
\end{split}\end{align}
When we sum \eqref{lu0.47} over $r=q$,
we get by \eqref{luw4.35} (cf. \cite[(8.3.58)]{MM07}),
\begin{align}\label{luw4.41}
\Big\langle \big(\nabla_{\frac{\p}{\p z_{i}}}^{X}J\big)
\frac{\p}{\p z_{q}},
\big(\nabla_{\frac{\p}{\p \ov{z}_{q}}}^{X}J\big)
\frac{\p}{\p\ov{z}_{i}}\Big\rangle
=\frac{1}{16}\big|\nabla^{X}J\big|^{2}.
\end{align}
By (\ref{luw4.13}) and (\ref{lu0.47}), we get
\begin{align}\label{lu0.48}
\Big\langle \big(\nabla_{\frac{\p}{\p z_{i}}}^{X}J\big)
\frac{\p}{\p z_{r}},
\big(\nabla_{\frac{\p}{\p \ov{z}_{i}}}^{X}J\big)
\frac{\p}{\p\ov{z}_{q}}+
\big(\nabla_{\frac{\p}{\p \ov{z}_{q}}}^{X}J\big)
\frac{\p}{\p\ov{z}_{i}}\Big\rangle =2\cJ_{ijr}
\big(2\cJ_{\ov{i}\,\ov{j}\ov{q}}-\cJ_{\ov{j}\,\ov{i}\ov{q}}\big).
\end{align}
By Lemma \ref{4.9}, (\ref{luw4.40}) and
$\big|\frac{\p}{\p z_{j}}\big|^{2}=\frac{1}{2}$, we obtain

\begin{align}\label{lu0.54b}\begin{split}
& \quad\Big\langle \big(\nabla_{\frac{\p}{\p z_{r}}}^{X}J\big)
\frac{\p}{\p z_{i}},
\big(\nabla_{\frac{\p}{\p \ov{z}_{q}}}^{X}J\big)
\frac{\p}{\p\ov{z}_{i}}\Big\rangle
\\
&=2\bigg[\Big\langle \big(\nabla_{\frac{\p}{\p z_{i}}}^{X}J
\big)\frac{\p}{\p z_{j}}, \frac{\p}{\p z_{r}}
\Big\rangle+
\Big\langle \big(\nabla_{\frac{\p}{\p z_{j}}}^{X}J\big)
\frac{\p}{\p z_{r}}, \frac{\p}{\p z_{i}}
\Big\rangle\bigg]
 \! \bigg[\Big\langle
\big(\nabla_{\frac{\p}{\p \ov{z}_{i}}}^{X}J\big)
\frac{\p}{\p \ov{z}_{j}},
\frac{\p}{\p \ov{z}_{q}}\Big\rangle
+\Big\langle \big(\nabla_{\frac{\p}{\p \ov{z}_{j}}}^{X}J
\big)\frac{\p}{\p \ov{z}_{q}},
\frac{\p}{\p \ov{z}_{i}}\Big\rangle\bigg]
 \\&=
4\cJ_{ijr}\big(\cJ_{\ov{i}\,\ov{j}\ov{q}}
-\cJ_{\ov{j}\,\ov{i}\ov{q}}\big).
\end{split}\end{align}
By taking the conjugation of (\ref{lu0.47}), we get
\begin{align}\label{lu0.49}
\Big\langle \big(\nabla_{\frac{\p}{\p z_{r}}}^{X}J\big)
\frac{\p}{\p z_{i}},
\big(\nabla_{\frac{\p}{\p \ov{z}_{i}}}^{X}J\big)
\frac{\p}{\p\ov{z}_{q}}\Big\rangle
= 2\cJ_{\ov{i}\,\ov{j}\ov{q}} \big(\cJ_{ijr} - \cJ_{jir}\big)
= 2\cJ_{ijr}\big(\cJ_{\ov{i}\,\ov{j}\ov{q}}
-\cJ_{\ov{j}\,\ov{i}\ov{q}}\big).
\end{align}
Substituting (\ref{luw4.35}), (\ref{luw4.41}), (\ref{lu0.48}),
(\ref{lu0.54b}) and (\ref{lu0.49}) into (\ref{lu0.46}) yields
\begin{align}\label{lu0.50a}
(d_{x}d_{y}I_{26})(0, 0) =
\frac{\sqrt{-1}}{8}\big|\nabla^{X}J\big|^{2}\omega(x_{0})
-\frac{2}{3}\cJ_{ijr}\big(5\cJ_{\ov{i}\,\ov{j}\ov{q}}
-4\cJ_{\ov{j}\,\ov{i}\ov{q}}\big)
dz_{r}\wedge d\ov{z}_{q}.
\end{align}

\subsection{Evaluation of $(d_{x}d_{y}I_{2})(0, 0)$:
part II}\label{luws4.3}
We evaluate now the contribution of $I_{22}$, $I_{24}$ in
$(d_{x}d_{y}I_{2})(0, 0)$.
The definitions of $\nabla^{X}\nabla^{X}J$ and $R^{TX}$
imply that for $U, V, W, Y\in TX$ (cf. \cite[(8.3.59)]{MM07}),
\begin{align}\label{luw4.55}\begin{split}
& \big(\nabla^{X}\nabla^{X}J\big)_{(U, V)}
-\big(\nabla^{X}\nabla^{X}J\big)_{(V, U)}
=\big[R^{TX}(U, V), J\big],   \\
& \Big\langle \big(\nabla^{X}\nabla^{X}J\big)_{(Y, U)}V,
W\Big\rangle
+\Big\langle\big(\nabla^{X}\nabla^{X}J\big)_{(Y, V)}W, U\Big\rangle
+\Big\langle \big(\nabla^{X}\nabla^{X}J\big)_{(Y, W)}U,
V\Big\rangle=0.
\end{split}\end{align}
Recall that \cite[(8.3.71)]{MM07},
\begin{equation}\label{lu0.58}
\sum_{|\alpha|=2}
\big(\partial^{\alpha}R^{L}\big)(\mR, \frac{\p}{\p\ov{z}_{i}})
\frac{Z^{\alpha}}{\alpha!}
=-\sqrt{-1}\pi\Big\langle\big(\nabla^{X}\nabla^{X}J\big)
_{(\mR, \mR)}\mR, \frac{\p}{\p\ov{z}_{i}}\Big\rangle
-\frac{2\pi}{3}\Big\langle R^{TX}(z, \ov{z})\mR,
\frac{\p}{\p \ov{z}_{i}}\Big\rangle.
\end{equation}
By (\ref{luw4.40}) and  (\ref{luw4.55}), we get
(cf. \cite[(8.3.61)]{MM07}): for $u_1,u_2, u_{3}\in T^{(1,0)}X$,
$\ov{v}_1,\ov{v}_2\in T^{(0,1)}X$,
\begin{equation}\label{bsy3.68}
\begin{split}
&(\nabla ^{X}\nabla ^{X}J)_{(u_1,u_2)}u_{3},\, (\nabla ^{X}
\nabla ^{X}J)_{(\ov{v}_1,\ov{v}_2)}u_{3}\in T^{(0,1)}X, \\
& (\nabla ^{X}\nabla ^{X}J)_{(u_1,\ov{v}_2)}u_{3}\in T^{(1,0)}X, \\
&2\sqrt{-1} \left\langle (\nabla ^{X}\nabla ^{X}J)_{(u_1,\ov{v}_1)}
u_2 ,  \ov{v}_2\right\rangle
= \left\langle  (\nabla ^{X}_{u_1}J) u_2,
(\nabla ^{X}_{\ov{v}_1}J)\ov{v}_2\right\rangle \, .
\end{split}
\end{equation}
In particular, we have
\begin{align}\label{lu0.59}\begin{split}
\Big\langle\big(\nabla^{X}\nabla^{X}J\big)_{(\ov{z}, z)}\ov{z},
\frac{\p}{\p\ov{z}_{i}}\Big\rangle&=0,\\
\Big\langle\big(\nabla^{X}\nabla^{X}J\big)_{(z, z)}u_{1},
\frac{\p}{\p\ov{z}_{i}}\Big\rangle &=
\Big\langle\big(\nabla^{X}\nabla^{X}J\big)_{(\ov{z}, \ov{z})}u_{1},
\frac{\p}{\p\ov{z}_{i}}\Big\rangle=0, \\
\Big\langle\big(\nabla^{X}\nabla^{X}J\big)_{(\ov{z}, z)}u_{1},
\frac{\p}{\p\ov{z}_{i}}\Big\rangle
&=\Big\langle\big(\nabla^{X}\nabla^{X}J\big)_{(z, \ov{z})}u_{1}
- [R^{TX}(z, \ov{z}), J] u_{1}, \frac{\p}{\p\ov{z}_{i}}\Big\rangle\\
&=\Big\langle\big(\nabla^{X}\nabla^{X}J\big)_{(z, \ov{z})}u_{1},
 \frac{\p}{\p\ov{z}_{i}}\Big\rangle.
\end{split}\end{align}
By (\ref{luw4.55}) and (\ref{lu0.59}), we get
\begin{align}\label{lu0.60}\begin{split}
\Big\langle\big(\nabla^{X}\nabla^{X}J\big)_{(z, \ov{z})}\ov{z},
\frac{\p}{\p\ov{z}_{i}}\Big\rangle
=&
\Big\langle \big[R^{TX}(z, \ov{z}), J\big]\ov{z},
\frac{\p}{\p \ov{z}_{i}} \Big\rangle
\\=&
-2\sqrt{-1}\Big\langle R^{TX}(z, \ov{z})\ov{z},
\frac{\p}{\p \ov{z}_{i}} \Big\rangle.
\end{split}\end{align}
By 
(\ref{luw4.55}) and (\ref{bsy3.68})
we get (cf. \cite[(8.3.62)]{MM07}),
\begin{align}\label{luw4.57}
\Big\langle\big(\nabla^{X}\nabla^{X}J\big)_{(u_{1}, u_{2})}\ov{z},
\frac{\p}{\p\ov{z}_{i}}\Big\rangle
=\frac{1}{2\sqrt{-1}}\Big\langle \big(\nabla_{u_{1}}^{X}J\big)u_{2},
\big(\nabla_{\ov{z}}^{X}J\big)\frac{\p}{\p\ov{z}_{i}}
-\big(\nabla_{\frac{\p}{\p \ov{z}_{i}}}^{X}J\big)\ov{z} \Big\rangle.
\end{align}
By (\ref{bsy3.68}), (\ref{lu0.59}),
(\ref{lu0.60}) and (\ref{luw4.57}), we get
\begin{multline}\label{luw4.58}
 -\pi\sqrt{-1}\Big\langle
\big(\nabla^{X}\nabla^{X}J\big)_{(\mR, \mR)}\mR,
\frac{\p}{\p\ov{z}_{i}}\Big\rangle
= -\frac{\pi}{2}\Big\langle \big(\nabla_{z}^{X}J\big)z,
3\big(\nabla_{\ov{z}}^{X}J\big)\frac{\p}{\p \ov{z}_{i}}-
\big(\nabla_{\frac{\p}{\p\ov{z}_{i}}}^{X}J\big)\ov{z}\Big\rangle\\
-2\pi\Big\langle R^{TX}(z, \ov{z})\ov{z},
\frac{\p}{\p\ov{z}_{i}}\Big\rangle
-\pi\sqrt{-1}\Big\langle
\big(\nabla^{X}\nabla^{X}J\big)_{(\ov{z}, \ov{z})}\ov{z},
\frac{\p}{\p\ov{z}_{i}}\Big\rangle.
\end{multline}
By (\ref{4.21}), (\ref{lu0.58}) and (\ref{luw4.58}), we get
\begin{align}\label{lu0.63a}\begin{split}
I_{22}(Z, Z')=&
-\bigg[\frac{\pi}{4}\Big\langle
\big(\nabla_{\frac{\p}{\p z_{j}}}^{X}J\big)\frac{\p}{\p z_{k}},
3\big(\nabla_{\frac{\p}{\p \ov{z}_{s}}}^{X}J\big)
\frac{\p}{\p \ov{z}_{i}}-
\big(\nabla_{\frac{\p}{\p\ov{z}_{i}}}^{X}J\big)
\frac{\p}{\p \ov{z}_{s}}\Big\rangle
 \\&\ \ \ \ \
+\frac{\pi}{3}\Big\langle R^{TX}\Big(\frac{\p}{\p z_{j}},
\frac{\p}{\p \ov{z}_{s}}\Big)\frac{\p}{\p z_{k}},
\frac{\p}{\p \ov{z}_{i}} \Big\rangle\bigg]
\Big(\cL^{-1}b_{i}z_{j}z_{k}\ov{z}_{s}\cP\Big)(Z, Z')
 \\&
-\frac{4\pi}{3}\Big\langle R^{TX}\Big(\frac{\p}{\p z_{j}},
\frac{\p}{\p\ov{z}_{s}}\Big)\frac{\p}{\p\ov{z}_{t}},
\frac{\p}{\p\ov{z}_{i}}\Big\rangle
\Big(\cL^{-1}b_{i}z_{j}\ov{z}_{s}\ov{z}_{t}\cP\Big)(Z, Z')
 \\&
-\frac{\pi}{2}\sqrt{-1}\bigg\{\cL^{-1}b_{i}
\Big\langle\big(\nabla^{X}\nabla^{X}J\big)_{(\ov{z}, \ov{z})}
\ov{z}, \frac{\p}{\p\ov{z}_{i}}\Big\rangle\cP\bigg\}(Z, Z').
\end{split}\end{align}
By (\ref{1.6}) and (\ref{1.6a}), we get
\begin{align}\label{luw4.59}\begin{split}
b_{i}z_{j}z_{k}\ov{z}_{s}\cP&=
b_{i}z_{j}z_{k}\big(\frac{b_{s}}{2\pi}+\ov{z}'_{s}\big)\cP
\\ &=
\Big(\frac{b_{i}b_{s}}{2\pi}z_{j}z_{k}
+\frac{\delta_{js}}{\pi}b_{i}z_{k}
+\frac{\delta_{ks}}{\pi}b_{i}z_{j}
+b_{i}z_{j}z_{k}\ov{z}'_{s}\Big)\cP.
\end{split}\end{align}
Thus, by Theorem \ref{lut1.1}, (\ref{1.6}), (\ref{1.6a})
and (\ref{luw4.59}), as in (\ref{4.42}), we get
\begin{align}\label{lu0.63}\begin{split}
&\Big(\cL^{-1}b_{i}z_{j}z_{k}\ov{z}_{s}\cP\Big)(Z, Z')
\\ &=
\Big(\frac{b_{i}b_{s}}{16\pi^{2}}z_{j}z_{k}
+\frac{\delta_{js}}{4\pi^{2}}b_{i}z_{k}
+\frac{\delta_{ks}}{4\pi^{2}}b_{i}z_{j}
+\frac{b_{i}}{4\pi}z_{j}z_{k}\ov{z}'_{s}\Big)\cP(Z, Z')
 \\&=
\frac{1}{4\pi^{2}}\Big(-\delta_{js}\delta_{ik}
+\pi\delta_{js}z_{k}\big(\ov{z}_{i}-\ov{z}'_{i}\big)-
\delta_{ks}\delta_{ij}+\pi\delta_{ks}z_{j}
\big(\ov{z}_{i}-\ov{z}'_{i}\big)
 \\ &\qquad\qquad -
\pi\delta_{ij}z_{k}\big(\ov{z}_{s}-\ov{z}'_{s}\big)
-\pi\delta_{ik}z_{j}\big(\ov{z}_{s}-\ov{z}'_{s}\big)
+\pi^{2}z_{j}z_{k}\big(\ov{z}_{i}-\ov{z}'_{i}\big)
\big(\ov{z}_{s}-\ov{z}'_{s}\big)\Big)\cP(Z, Z')
 \\&\quad+
\frac{1}{4\pi}\Big(-2\delta_{ij}z_{k}-2\delta_{ik}z_{j}
+2\pi z_{j}z_{k}\big(\ov{z}_{i}-\ov{z}'_{i}\big)\Big)\ov{z}'_{s}
\cP(Z, Z').
\end{split}\end{align}
By (\ref{luw4.11}) and (\ref{lu0.63}), note that the last line of (\ref{lu0.63})
has also the term $z\ov{z}'$, we get
\begin{multline}\label{luw4.60}
\Big(\cL^{-1}b_{i}z_{j}z_{k}\ov{z}_{s}\cP\Big)(Z, Z')
 \approx  - \Big[\frac{1}{4\pi^{2}}\big(\delta_{js}\delta_{ik}
+\delta_{ks}\delta_{ij}\big)\\
+\frac{1}{4\pi}\big(\delta_{js}z_{k}\ov{z}'_{i}
+\delta_{ks}z_{j}\ov{z}'_{i}
+\delta_{ij}z_{k}\ov{z}'_{s}
+\delta_{ik}z_{j}\ov{z}'_{s}\big)\Big]\cP(Z, Z').
\end{multline}
Again by (\ref{1.6}) and (\ref{1.6a}),
\begin{equation}\label{luw4.61}
\begin{split}
&b_{i}z_{j}\ov{z}_{s}\ov{z}_{t}\cP =
b_{i}z_{j}\Big(\frac{b_{s}}{2\pi}+\ov{z}'_{s}\Big)
\Big(\frac{b_{t}}{2\pi}+\ov{z}'_{t}\Big)\cP
\\ &=
\Biggl[\frac{b_{i}b_{s}b_{t}}{4\pi^{2}}z_{j}
+\frac{1}{2\pi^{2}}\big(\delta_{js}b_{i}b_{t}
+\delta_{jt}b_{i}b_{s}\big) +
\frac{b_{i}}{2\pi}\big(b_{s}z_{j}+2\delta_{js}\big)\ov{z}'_{t}
+\frac{b_{i}}{2\pi}\big(b_{t}z_{j}+2\delta_{jt}\big)\ov{z}'_{s}
+b_{i}z_{j}\ov{z}'_{s}\ov{z}'_{t}\Biggr]\cP.
\end{split}
\end{equation}
By Theorem \ref{lut1.1} and (\ref{luw4.61}), we get
\begin{multline}\label{luw4.63}
\cL^{-1}b_{i}z_{j}\ov{z}_{s}\ov{z}_{t}\cP
 = \Biggl[\frac{b_{i}b_{s}b_{t}}{48\pi^{3}}z_{j}
+\frac{1}{16\pi^{3}}\big(\delta_{js}b_{i}b_{t}
+\delta_{jt}b_{i}b_{s}\big)
 \ + \frac{b_{i}}{16\pi^{2}}\big(b_{s}z_{j}
+4\delta_{js}\big)\ov{z}'_{t}\\
+\frac{b_{i}}{16\pi^{2}}\big(b_{t}z_{j}+4\delta_{jt}\big)\ov{z}'_{s}
+\frac{b_{i}}{4\pi}z_{j}\ov{z}'_{s}\ov{z}'_{t}\Biggr]\cP.
\end{multline}
By (\ref{1.6}), (\ref{1.6a}) and (\ref{luw4.11}), we get
\begin{align}\label{luw4.64}
\Big(\frac{b_{i}}{16\pi^{2}}\big(b_{s}z_{j}
+4\delta_{js})\ov{z}'_{t}\cP\Big)(Z, Z')
\approx \frac{1}{4\pi}\big(\delta_{js}\ov{z}_{i}
-\delta_{ij}\ov{z}_{s}\big)\ov{z}'_{t}\cP(Z, Z').
\end{align}
By (\ref{luw4.52}), (\ref{luw4.63}) and (\ref{luw4.64}), we obtain
\begin{equation}\label{luw4.65}
\begin{split}
\Big(d_{x}d_{y}\big(\cL^{-1}b_{i}z_{j}\ov{z}_{s}\ov{z}_{t}
\cP\big)\Big)(0, 0)
&=\frac{1}{4\pi}\big(\delta_{js}d\ov{z}_{i}
-\delta_{ij}d\ov{z}_{s}\big)\wedge d\ov{z}_{t}
+\frac{1}{4\pi}\big(\delta_{jt}d\ov{z}_{i}
-\delta_{ij}d\ov{z}_{t}\big)\wedge d\ov{z}_{s}\\
&=\frac{1}{4\pi}\big(\delta_{js}d\ov{z}_{i}
\wedge d\ov{z}_{t}+\delta_{jt}d\ov{z}_{i}\wedge d\ov{z}_{s}\big).
\end{split}
\end{equation}
Finally, by Theorem \ref{lut1.1}, (\ref{1.6}) and (\ref{1.6a}),
as in (\ref{luw4.51}), we get
\begin{align}\label{luw4.66}\begin{split}
 &\Big(\cL^{-1}b_{i}\ov{z}_{j}\ov{z}_{s}\ov{z}_{t}\cP\Big)(Z, Z')
 =\Big(\cL^{-1}b_{i}\big(\frac{b_{j}}{2\pi}+\ov{z}'_{j}\big)
\big(\frac{b_{s}}{2\pi}+\ov{z}'_{s}\big)\big(\frac{b_{t}}{2\pi}
+\ov{z}'_{t}\big)\cP\Big)(Z, Z')\\
&\hspace{20mm}\sim
\bigg\{\cL^{-1}b_{i}\Big[\frac{b_{j}b_{s}b_{t}}{8\pi^{3}}
+\frac{1}{4\pi^{2}}\big(b_{j}b_{s}\ov{z}'_{t}
+b_{j}b_{t}\ov{z}'_{s}+b_{s}b_{t}\ov{z}'_{j}\big)\Big]\cP
\bigg\}(Z, Z') \\
&\hspace{20mm}\approx 0.
\end{split}\end{align}

\noindent
From (\ref{3.9}), (\ref{lu0.63a}), (\ref{luw4.60}), (\ref{luw4.65})
and (\ref{luw4.66}), we obtain
\begin{multline}\label{lu0.74}
(d_{x}d_{y}I_{22})(0, 0) =
\Bigg\{\frac{1}{8\pi}\Big\langle
\big(\nabla_{\frac{\p}{\p z_{j}}}^{X}J\big)\frac{\p}{\p z_{i}},
\big(\nabla_{\frac{\p}{\p \ov{z}_{j}}}^{X}J\big)
\frac{\p}{\p \ov{z}_{i}}
+ \big(\nabla_{\frac{\p}{\p \ov{z}_{i}}}^{X}J\big)
\frac{\p}{\p \ov{z}_{j}} \Big\rangle \\
+\frac{1}{12\pi}\Big\langle R^{TX}\Big(\frac{\p}{\p z_{i}},
\frac{\p}{\p\ov{z}_{i}}\Big)\frac{\p}{\p z_{j}}
+  R^{TX}\Big(\frac{\p}{\p z_{j}}, \frac{\p}{\p\ov{z}_{i}}\Big)
\frac{\p}{\p z_{i}},
\frac{\p}{\p \ov{z}_{j}}\Big\rangle\Bigg\}
\big(-2\pi\sqrt{-1}\big)\omega(x_{0}) \\
+
\frac{1}{8}\bigg[\Big\langle \big(\nabla_{\frac{\p}{\p z_{j}}}^{X}J
\big)\frac{\p}{\p z_{r}},
\big(\nabla_{\frac{\p}{\p \ov{z}_{j}}}^{X}J\big)
\frac{\p}{\p \ov{z}_{q}}
+ \big(\nabla_{\frac{\p}{\p \ov{z}_{q}}}^{X}J\big)
\frac{\p}{\p \ov{z}_{j}} \Big\rangle\\
+ \Big\langle \big(\nabla_{\frac{\p}{\p z_{r}}}^{X}J\big)
\frac{\p}{\p z_{j}},
\big(\nabla_{\frac{\p}{\p \ov{z}_{j}}}^{X}J\big)
\frac{\p}{\p \ov{z}_{q}}
+\big(\nabla_{\frac{\p}{\p \ov{z}_{q}}}^{X}J\big)
\frac{\p}{\p \ov{z}_{j}} \Big\rangle \bigg]
dz_{r}\wedge d\ov{z}_{q}\\
+ \frac{1}{6}\Big\langle R^{TX}\Big(\frac{\p}{\p z_{j}},
\frac{\p}{\p\ov{z}_{j}}\Big)\frac{\p}{\p z_{r}}
 + R^{TX}\Big(\frac{\p}{\p z_{r}},
\frac{\p}{\p\ov{z}_{j}}\Big)\frac{\p}{\p z_{j}},
\frac{\p}{\p \ov{z}_{q}} \Big\rangle
dz_{r}\wedge d\ov{z}_{q} \\
+\frac{1}{3}\Big\langle R^{TX}\Big(\frac{\p}{\p z_{j}},
\frac{\p}{\p\ov{z}_{j}}\Big)\frac{\p}{\p \ov{z}_{r}}
+  R^{TX}\Big(\frac{\p}{\p z_{j}},
\frac{\p}{\p\ov{z}_{r}}\Big)\frac{\p}{\p \ov{z}_{j}},
\frac{\p}{\p \ov{z}_{q}}\Big\rangle
d\ov{z}_{r}\wedge d\ov{z}_{q}.
\end{multline}
Note that by (\ref{luw4.36}),
\begin{align}\label{luw4.68}\begin{split}
&  R^{TX}\Big(\frac{\p}{\p z_{j}}, \frac{\p}{\p\ov{z}_{j}}\Big)
\frac{\p}{\p z_{r}}  =
R^{TX}\Big(\frac{\p}{\p z_{r}}, \frac{\p}{\p\ov{z}_{j}}\Big)
\frac{\p}{\p z_{j}}
+ R^{TX}\Big(\frac{\p}{\p z_{j}}, \frac{\p}{\p z_{r}}\Big)
\frac{\p}{\p \ov{z}_{j}}.
\end{split}\end{align}
By  (\ref{luw4.36}), (\ref{luw4.35}), (\ref{luw4.41}),
(\ref{lu0.48}), (\ref{lu0.54b}),
(\ref{lu0.49}), (\ref{lu0.74}) and (\ref{luw4.68}), we get
\begin{multline}\label{lu0.75}
 (d_{x}d_{y}I_{22})(0, 0)=
-\sqrt{-1}\bigg[\frac{5}{96}\big|\nabla^{X}J\big|^{2}+
\frac{1}{3}\Big\langle R^{TX}\Big(\frac{\p}{\p z_{j}},
\frac{\p}{\p\ov{z}_{i}}\Big)\frac{\p}{\p z_{i}},
\frac{\p}{\p \ov{z}_{j}} \Big\rangle \bigg]\omega(x_{0})\\
+ \frac{1}{4}\cJ_{jir} \big(5 \cJ_{\ov{j}\,\ov{i}\ov{q}}
 -4 \cJ_{\ov{i}\,\ov{j}\ov{q}}\big)
dz_{r}\wedge d\ov{z}_{q} \\
+ \frac{1}{6}\Big\langle R^{TX}\Big(\frac{\p}{\p z_{j}},
\frac{\p}{\p z_{r}}\Big)\frac{\p}{\p \ov{z}_{j}}
+  2 R^{TX}\Big(\frac{\p}{\p z_{r}},
\frac{\p}{\p\ov{z}_{j}}\Big)\frac{\p}{\p z_{j}},
\frac{\p}{\p \ov{z}_{q}} \Big\rangle
dz_{r}\wedge d\ov{z}_{q} \\
+ \frac{1}{3}\Big\langle R^{TX}\Big(\frac{\p}{\p z_{j}},
\frac{\p}{\p\ov{z}_{j}}\Big)\frac{\p}{\p \ov{z}_{r}}
+ R^{TX}\Big(\frac{\p}{\p z_{j}}, \frac{\p}{\p\ov{z}_{r}}\Big)
\frac{\p}{\p \ov{z}_{j}}, \frac{\p}{\p \ov{z}_{q}}\Big\rangle
d\ov{z}_{r}\wedge d\ov{z}_{q}.
\end{multline}
By (\ref{luw4.55}), (\ref{bsy3.68}) and (\ref{lu0.59}),
we get (cf. \cite[(2.28)]{MM08a}),
\begin{align} \label{4.49}\begin{split}
\Big\langle \big(\nabla^{X}\nabla^{X}J\big)_{(\mR, \mR)}
\frac{\p}{\p z_{i}}, \frac{\p}{\p\ov{z}_{i}}\Big\rangle
&= 2\Big\langle \big(\nabla^{X}\nabla^{X}J\big)_{(z, \ov{z})}
\frac{\p}{\p z_{i}},\frac{\p}{\p\ov{z}_{i}}\Big\rangle\\
&= -\sqrt{-1} \Big\langle \big(\nabla_{z}^{X}J\big)
\frac{\p}{\p z_{i}}, \big(\nabla_{\ov{z}}^{X}J\big)
\frac{\p}{\p\ov{z}_{i}}\Big\rangle.
\end{split}\end{align}
By (\ref{luw4.33}) and (\ref{lu0.55}), we get
\begin{align}\label{4.50}
\Big(\cL^{-1}\cP^{\bot}z_{s}\ov{z}_{t}\cP\Big)(Z, Z')=
-\frac{1}{4\pi^{2}}\Big(\delta_{st}-\pi z_{s}\big(\ov{z}_{t}
-\ov{z}'_{t}\big)\Big)\cP(Z, Z'). \end{align}
\comment{
From (\ref{4.21}), (\ref{4.49}) and (\ref{4.50}), we obtain
\begin{align}\label{lu0.53}
I_{24}(Z, Z')=&
\frac{1}{2\pi} \Big\langle
\big(\nabla_{\frac{\p}{\p z_{s}}}^{X}J\big)
\frac{\p}{\p z_{i}},
\big(\nabla_{\frac{\p}{\p \ov{z}_{t}}}^{X}J\big)
\frac{\p}{\p\ov{z}_{i}}\Big\rangle
\Big(\delta_{st}-\pi z_{s}\big(\ov{z}_{t}-\ov{z}'_{t}\big)\Big)
\cP(Z, Z').
\end{align}
}
By (\ref{3.9}), (\ref{luw4.35}), (\ref{lu0.54b}),
(\ref{4.49}) and (\ref{4.50}), we obtain
\begin{equation}\label{lu0.54a}
\begin{split}
(d_{x}d_{y}I_{24})(0, 0)&=
\Big\langle \big(\nabla_{\frac{\p}{\p z_{j}}}^{X}J\big)
\frac{\p}{\p z_{i}},
\big(\nabla_{\frac{\p}{\p \ov{z}_{j}}}^{X}J\big)
\frac{\p}{\p\ov{z}_{i}}
\Big\rangle (-\sqrt{-1})\omega(x_{0}) \\
&\qquad\qquad\qquad+ \frac{1}{2} \Big\langle
\big(\nabla_{\frac{\p}{\p z_{s}}}^{X}J\big)
\frac{\p}{\p z_{i}},
\big(\nabla_{\frac{\p}{\p \ov{z}_{t}}}^{X}J\big)
\frac{\p}{\p\ov{z}_{i}}\Big\rangle
dz_{s}\wedge d\ov{z}_{t}\\
&= -\frac{\sqrt{-1}}{8}\big|\nabla^{X}J\big|^{2}\omega(x_{0})
+2\cJ_{ijr}\big(\cJ_{\ov{i}\,\ov{j}\ov{q}}
-\cJ_{\ov{j}\,\ov{i}\ov{q}}\big) dz_{r}\wedge d\ov{z}_{q}.
\end{split}
\end{equation}
\comment{
Note that by (\ref{luw4.36}),
\begin{align}\label{luw4.68}\begin{split}
&  R^{TX}\Big(\frac{\p}{\p z_{j}}, \frac{\p}{\p\ov{z}_{j}}\Big)
\frac{\p}{\p z_{r}}  =
R^{TX}\Big(\frac{\p}{\p z_{r}}, \frac{\p}{\p\ov{z}_{j}}\Big)
\frac{\p}{\p z_{j}}
+ R^{TX}\Big(\frac{\p}{\p z_{j}}, \frac{\p}{\p z_{r}}\Big)
\frac{\p}{\p \ov{z}_{j}}.
\end{split}\end{align}
}
Combining (\ref{4.22}),  (\ref{luw4.36}), (\ref{lu0.62a}),
(\ref{lu0.66a}),  (\ref{luw4.35}),
(\ref{lu0.41a}), (\ref{lu0.50a}), (\ref{luw4.68}),
(\ref{lu0.75})  and (\ref{lu0.54a}), we obtain
\begin{multline}\label{lu0.78c}
-(d_{x}d_{y}I_{2})(0, 0)=
\sqrt{-1}\Big\langle R^{TX}\Big(\frac{\p}{\p z_{j}},
\frac{\p}{\p \ov{z}_{i}}\Big)
\frac{\p}{\p z_{i}}, \frac{\p}{\p\ov{z}_{j}}
\Big\rangle \,\omega(x_{0}) \\
-\frac{1}{2}\Big\langle R^{TX}\Big(\frac{\p}{\p z_{j}},
\frac{\p}{\p z_{r}}\Big)
\frac{\p}{\p \ov{z}_{j}}, \frac{\p}{\p\ov{z}_{q}}\Big\rangle
dz_{r}\wedge d\ov{z}_{q}
 - \frac{1}{12}\cJ_{jir}\big(\cJ_{\ov{j}\,\ov{i}\ov{q}}
+4\cJ_{\ov{i}\,\ov{j}\ov{q}}\big)
dz_{r}\wedge d\ov{z}_{q}.
\end{multline}
By Lemma \ref{4.9}, (\ref{luw4.40}), (\ref{lu0.47}), (\ref{lu0.54b}),
(\ref{lu0.49}), (\ref{luw4.55}) and (\ref{luw4.57}),
we get (cf. \cite[(8.3.63)]{MM07})
\begin{align}\label{lu0.78b}\begin{split}
\Big\langle R^{TX}&\Big(\frac{\p}{\p z_{j}}, \frac{\p}{\p z_{r}}\Big)
\frac{\p}{\p \ov{z}_{j}}, \frac{\p}{\p\ov{z}_{q}}\Big\rangle
=\frac{\sqrt{-1}}{2} \Big\langle \Big[R^{TX}\Big(\frac{\p}{\p z_{j}},
\frac{\p}{\p z_{r}}\Big), J\Big]
\frac{\p}{\p \ov{z}_{j}}, \frac{\p}{\p\ov{z}_{q}}\Big\rangle\\
&=\frac{\sqrt{-1}}{2}\Big\langle
\Big[\big(\nabla^{X}\nabla^{X}J\big)_{(\frac{\p}{\p z_{j}},
\frac{\p}{\p z_{r}})}
-\big(\nabla^{X}\nabla^{X}J\big)_{(\frac{\p}{\p z_{r}},
\frac{\p}{\p z_{j}})}\Big]
\frac{\p}{\p \ov{z}_{j}},
\frac{\p}{\p \ov{z}_{q}}\Big\rangle  \\
&=\frac{1}{4}\Big\langle \big(\nabla_{\frac{\p}{\p z_{j}}}^{X}J\big)
\frac{\p}{\p z_{r}}
-\big(\nabla_{\frac{\p}{\p z_{r}}}^{X}J\big)\frac{\p}{\p z_{j}},\
\big(\nabla_{\frac{\p}{\p\ov{z}_{j}}}^{X}J\big)
\frac{\p}{\p\ov{z}_{q}}
-\big(\nabla_{\frac{\p}{\p\ov{z}_{q}}}^{X}J\big)
\frac{\p}{\p\ov{z}_{j}}\Big\rangle \\
& =\frac{1}{2}\cJ_{jri}\cJ_{\ov{j}\ov{q}\ov{i}}.
\end{split}\end{align}
Substituting (\ref{lu0.78b}) into (\ref{lu0.78c}) yields
\begin{equation}\label{lu0.78a}
(d_{x}d_{y}I_{2})(0, 0)=
-\sqrt{-1}\Big\langle R^{TX}\Big(\frac{\p}{\p z_{j}},
\frac{\p}{\p \ov{z}_{i}}\Big)
\frac{\p}{\p z_{i}}, \frac{\p}{\p\ov{z}_{j}}\Big\rangle
\omega(x_{0})
+\,\frac{1}{3}\cJ_{jir}\big(\cJ_{\ov{j}\,\ov{i}\ov{q}}
+ \cJ_{\ov{i}\,\ov{j}\ov{q}}\big)
dz_{r}\wedge d\ov{z}_{q}.
\end{equation}
\subsection{Proof of Theorem \ref{t3.1}}\label{luws4.4}
By (\ref{luw2.2}) and (\ref{lu0.78a}), we get
\begin{align}\label{lu0.98}\begin{split}
(d_{x}d_{y}I_{4})(0, 0)
=&(d_{x}d_{y}I_{2})(0, 0).
\end{split}\end{align}
Substituting  (\ref{lu0.32a}), (\ref{lu0.35a}), (\ref{lu0.21a}),
(\ref{lu0.28a}), (\ref{lu0.78a}) and (\ref{lu0.98}) into (\ref{lu0.5}),
we finally obtain
\begin{multline}\label{4.86}
\big(d_{x}d_{y}\cF_{2}\big)(0, 0)=
-\frac{6}{9} \cJ_{jir}\big(\cJ_{\ov{i}\,\ov{j}\ov{q}}
+ \cJ_{\ov{j}\,\ov{i}\ov{q}}\big)dz_{r}\wedge d\ov{z}_{q}
+ 2 (d_{x}d_{y}I_{2})(0, 0)\\
=-2\sqrt{-1}\Big\langle R^{TX}\Big(\frac{\p}{\p z_{j}},
\frac{\p}{\p\ov{z}_{i}}\Big)
\frac{\p}{\p z_{i}}, \frac{\p}{\p\ov{z}_{j}}\Big\rangle
\, \omega(x_{0}).
\end{multline}
\comment{
\begin{multline}\label{4.86}
\big(d_{x}d_{y}\cF_{2}\big)(0, 0)=
\Big(-\frac{1}{9}-\frac{1}{9} - \frac{2}{9}-\frac{2}{9}\Big)
\cJ_{jir}\big(\cJ_{\ov{i}\,\ov{j}\ov{q}}
+ \cJ_{\ov{j}\,\ov{i}\ov{q}}\big)
dz_{r}\wedge d\ov{z}_{q}\\
-2\sqrt{-1}\Big\langle R^{TX}\Big(\frac{\p}{\p z_{j}},
\frac{\p}{\p\ov{z}_{i}}\Big)
\frac{\p}{\p z_{i}}, \frac{\p}{\p\ov{z}_{j}}\Big\rangle
\, \omega(x_{0})
+\frac{2}{3}\cJ_{jir}\big(\cJ_{\ov{j}\,\ov{i}\ov{q}}
+\cJ_{\ov{i}\,\ov{j}\ov{q}}\big)
dz_{r}\wedge d\ov{z}_{q} \\
=-2\sqrt{-1}\Big\langle R^{TX}\Big(\frac{\p}{\p z_{j}},
\frac{\p}{\p\ov{z}_{i}}\Big)
\frac{\p}{\p z_{i}}, \frac{\p}{\p\ov{z}_{j}}\Big\rangle
\, \omega(x_{0}).
\end{multline}
}
By \cite[Theorem\,8.3.4, Lemma\,8.3.10]{MM07},
\begin{align}\label{luw4.74}
8\Big\langle R^{TX}\Big(\frac{\p}{\p z_{j}},
\frac{\p}{\p\ov{z}_{i}}\Big)
\frac{\p}{\p z_{i}}, \frac{\p}{\p\ov{z}_{j}}\Big\rangle
=r^{X}
+\frac{1}{4}\big|\nabla^{X}J\big|^{2}
= 8\pi \bb_{1}(x_{0}).
\end{align}
The identities \eqref{4.86} and \eqref{luw4.74} yield
Theorem \ref{t3.1}. This concludes the proof
of Theorem \ref{t0.1}.

\end{document}